\documentclass[12pt,oneside]{amsart}
\usepackage{amsmath,amssymb,cite,mathrsfs,tikz-cd}
\usepackage[all]{xy}
\usepackage{typearea} 
\usepackage{hyperref} 
\usepackage{soul,mathtools}

\newcounter{maincounter}
\numberwithin{maincounter}{section}
\numberwithin{equation}{section}

\newtheorem{lemma}[maincounter]{Lemma}
\newtheorem{proposition}[maincounter]{Proposition}
\newtheorem{corollary}[maincounter]{Corollary}
\newtheorem{remark}[maincounter]{Remark}
\newtheorem{theorem}[maincounter]{Theorem}
\newtheorem{conjecture}[maincounter]{Conjecture}
\newtheorem{conjecturerelmm*}{Conjecture RelMM(g)}

\def\AA{\mathbb{A}}

\def\NN{\mathbb{N}}
\def\RR{\mathbb{R}}

\def\CC{\mathbb{C}}
\def\ZZ{\mathbb{Z}}

\def\QQ{\mathbb{Q}}

\def\GG{\mathbb{G}}

\newcommand{\cal}{\mathcal}

\newcommand{\cA}{\cal{A}}

\newcommand{\cB}{\cal{B}}

\newcommand{\anE}{\mathrm{an}}
\newcommand{\an}[1]{{#1}^{\anE}}

\newcommand{\IG}{{\GG}}
\newcommand{\IC}{{\CC}}
\newcommand{\IR}{{\RR}}

\newcommand{\IQbar}{{\overline{\QQ}}}

\newcommand{\IZ}{{\ZZ}}
\newcommand{\IN}{{\NN}}
\newcommand{\IA}{{\AA}}
\newcommand{\IQ}{{\QQ}}

\newcommand{\cE}{{\mathcal E}}

\newcommand{\cY}{{\mathcal Y}}

\newcommand{\mat}[2]{{\rm Mat}_{#1}({#2})}
\newcommand{\ssm}{\setminus}

\newcommand{\trans}{\top}

\renewcommand{\subset}{\subseteq} 

\newcommand{\pullbackcorner}[1][dr]{\save*!/#1-1.7pc/#1:(-1.5,1.5)@^{|-}\restore}

\newcommand{\leftpullback}[1][dl]{\save*!/#1-1.2pc/#1:(-1,1)@^{|-}\restore}

\begin{document}

\title{Degeneracy Loci in the Universal Family of Abelian Varieties}
\author{Ziyang Gao and Philipp Habegger}
\address{IAZD (Institute of Algebra, Number Theory and Discrete Mathematics), Leibniz University Hannover, Welfengarten 1, 30167 Hannover, Germany}
\email{ziyang.gao@math.uni-hannover.de}
\address{Department of Mathematics and Computer Science, University of Basel, Spiegelgasse 1, 4051 Basel, Switzerland}
\email{philipp.habegger@unibas.ch}
\begin{abstract}
  Recent developments  on the
  uniformity of the number of
  rational points on curves and subvarieties in a moving abelian
  variety rely on the geometric concept of the degeneracy locus.
  The first-named author
  investigated the degeneracy locus in certain mixed
  Shimura varieties. In this expository note we revisit some of these results
    while minimizing the use of mixed Shimura varieties while working
    in a family of principally polarized abelian varieties. We also
  explain their relevance for applications in diophantine geometry.
\end{abstract}
\maketitle
\tableofcontents

\section{Introduction}

The goal of this expository note is to reprove some arguments in \cite{Gao:SNS17, GaoBettiRank}, especially regarding the degeneracy loci, \textit{with a minimal use of the language of mixed Shimura varieties}.

\smallskip

With a view towards application, we will work in the following setup.
Let $\mathfrak{A}_g \rightarrow \mathbb{A}_g$ be the
universal family of principally polarized $g$-dimensional abelian
varieties with level-$\ell$-structure for some $\ell\ge 3$. Then
$\mathfrak{A}_g$ carries the structure of a geometrically irreducible
quasi-projective variety defined over a number field.

Let $X$ be an irreducible closed subvariety of $\mathfrak{A}_g$. In
\cite{GaoBettiRank}, the first-named author defined the \textit{$t$-th
degeneracy locus} $X^{\mathrm{deg}}(t)$ for each $t \in \IZ$; we refer
to $\mathsection$\ref{sec:critnondeg} for a definition in our setting.
By definition, $X^{\mathrm{deg}}(t)$ is an at most countably infinite
union of Zariski closed subsets of $X$.
Yet $X^{\mathrm{deg}}(t)$ is Zariski closed in $X$, see
\cite[Theorem~1.8]{GaoBettiRank}.

The definition of  $X^{\mathrm{deg}}(t)$ involves
 \textit{bi-algebraic subvarieties} of $\mathfrak{A}_g$
and $\IA_g$; bi-algebraic subvarieties are 
explained in the beginning of $\mathsection$\ref{sec:bialgebraic}.
 Ullmo and Yafaev characterized~\cite{UYweaklyspecial} bi-algebraic subvarieties
of $\IA_g$ as the \textit{weakly special subvarieties} of $\IA_g$,
when we view $\IA_g$ as
a Shimura variety.
The first-named author \cite[Corollary~8.3]{GaoCrelle}
showed that the bi-algebraic subvarieties of $\mathfrak{A}_g$ are
precisely the \textit{weakly special} subvarieties, when we view
$\mathfrak{A}_g$ as a mixed Shimura variety, see
\cite[Definition~4.1(b)]{Pink05}. Then by some computation involving
mixed Shimura varieties, a geometric characterization of bi-algebraic
subvarieties of $\mathfrak{A}_g$ is given by
\cite[Proposition~1.1]{GaoCrelle}.

In the current paper we revisit the geometric description of a class of
bi-algebraic subvarieties of $\mathfrak{A}_g$. This is done in
Proposition~\ref{prop:monodromyaction}. Instead of obtaining a full
characterization, as in work of the first-named author,
we prove a slightly weaker result which is
sufficient for several applications in diophantine geometry.

A key tool in the proof of this proposition is already present in
the first-named author's work~\cite[$\mathsection$8]{GaoCrelle} as
well as Bertrand's overview of Manin's Theorem of the Kernel
\cite{Bertrand:Manin2020}. This tool is Andr\'e's normality theorem
for a variation of mixed Hodge structures~\cite{Andre92}.

We follow the ideas of \cite{GaoBettiRank}, in particular Theorem
8.1~\textit{loc.cit.} and derive a necessary
condition for $X^{\mathrm{deg}}(t)$ to be sufficiently large in $X$. The corresponding
result is stated here in Proposition~\ref{prop:deg} and relies on the
geometric structure result Proposition~\ref{prop:monodromyaction}.

Two values of $t$ are of particular interest for recent applications
to diophantine geometry.

In $\mathsection$\ref{sec:fiberpower} we emphasize
the case $t=0$.
The zeroth degeneracy locus $X^{\mathrm{deg}}(0)$ is of crucial importance in the
recent proof of the Uniform Mordell--Lang Conjecture \cite{DGHUnifML,
  KuehneUnifMM, UnifML}.
The mixed Ax--Schanuel
Theorem~\cite{GaoMixedAS} for the universal family of abelian
varieties
links the concept of non-degeneracy, in the sense of
\cite[Definition~1.5]{DGHUnifML}, 
with the size of $X^{\mathrm{deg}}(0)$ in $X$.
We 
refer to recent work of Bl\'azquez-Sanz--Casale--Freitag--Nagloo
\cite{BCFN:AxSchanuelI}
for a differential algebraic approach to the Ax--Schanuel Theorem.
More precisely, we replace $\mathfrak{A}_{g}$ by its
$m$-fold fiber power $\mathfrak{A}_{g}^{[m]}$ over $\IA_g$ for some
$m\in\IN=\{1,2,3,\ldots\}$ and consider
$\mathfrak{A}_g^{[m]}\subset\mathfrak{A}_{gm}$.  For the
Uniform Mordell--Lang Conjecture~\cite{DGHUnifML} for curves of genus $g\ge
2$, the subvariety $X\subset\mathfrak{A}_g^{[m]}$ is the image under
the \textit{Faltings--Zhang morphism} of the $(m+1)$-fold
 fiber
power of a suitable family of smooth projective curves of genus $g$.
Corollary~\ref{cor:fiberpowercurve} yields a sufficiently strong
statement to ensure that the relevant $X$ arising in \cite{DGHUnifML}
is non-degenerate. This corollary is a special case of \cite[Theorem~1.3]{GaoBettiRank}. But the theorem on which it is based,
Theorem~\ref{thm:deg0}, applies to more general subvarieties of
$\mathfrak{A}_g^{[m]}$ and is new.

As explained in Remark~\ref{rmk:bettirank}, the mixed Ax--Schanuel
Theorem for the universal family implies that if $X$ fails to be
non-degenerate, then $X^{\mathrm{deg}}(0)$ contains a non-empty open
subset of $X^{\mathrm{an}}$; here $X^{\mathrm{an}}$ means the
analytification of $X$ and the topology in consideration is the
Euclidean topology. (This and a converse claim is contained in
\cite[Theorem~1.7]{GaoBettiRank}.) Zariski closedness of
$X^{\mathrm{deg}}(0)$ is not logically required in the context of
\cite{DGHUnifML, KuehneUnifMM, UnifML}.

As observed in \cite[$\mathsection
11$]{GaoBettiRank},  $X^{\mathrm{deg}}(1)$ is linked to the
relative Manin--Mumford Conjecture; see $\mathsection$\ref{conj:relmmXdeg1dense} for a
formulation of this conjecture and a brief history. 
Our Corollary~\ref{cor:relmmconditional}
 contains an application of the structure
result, Proposition~\ref{prop:monodromyaction}, to reduce the relative
Manin--Mumford Conjecture for $\mathfrak{A}_g$ to
Conjecture~\ref{conj:relmmXdeg1dense}. This is in the spirit of
\cite[Proposition 11.2]{GaoBettiRank}. We plan to address
Conjecture~\ref{conj:relmmXdeg1dense} in future work.
In the current paper,  we restrict ourselves to the
family of principally polarized abelian varieties.
Again we avoid the language of mixed Shimura varieties. However,
our proof of Theorem~\ref{thm:relmminduction}, and so ultimately Corollary~\ref{cor:relmmconditional},
requires the Zariski closedness of $X^{\mathrm{deg}}(1)$ in $X$, which was 
 proved using the theory of mixed Shimura varieties~\cite[Theorem~1.8]{GaoBettiRank}. We do not reprove Zariski closedness in the current
paper.


\smallskip

Section~\ref{sec:prelimabsch} below contains some preliminaries on abelian schemes in
characteristic $0$.



\smallskip



\subsection*{Acknowledgements} The authors would like to thank
Gabriel Dill for discussions on abelian schemes and Lemma~\ref{lem:endabscheme}.
ZG has  received funding from the European Research Council (ERC) under the European Union’s Horizon 2020 research and innovation programme (grant agreement n$^{\circ}$ 945714). 
PH has received funding
from the Swiss National Science Foundation project n$^\circ$
200020\_184623.

\section{Preliminaries on Abelian Schemes}
\label{sec:prelimabsch}

Let $S$ be a smooth irreducible quasi-projective variety defined over
an algebraically closed subfield of 
$\IC$. By abuse of notation we often consider our varieties as defined
over $\IC$.
Let $\pi\colon \cA\rightarrow S$ be an abelian scheme of
relative dimension $g\ge 1$. For $s\in S$ we write $\cA_s$ for the abelian
variety over $\IC(s)$. More generally, for a morphism $T\rightarrow S$
of schemes we let $\cA_T$ denote the fiber power $\cA\times_S T$. Let $\eta\in S$ denote the generic point. 

Let $\mathrm{End}(\cA/S)$ denote the group of endomorphisms of the
abelian scheme $\cA\rightarrow S$. It is a finitely generated free
abelian group. For all $s\in S$ let $\varphi_s$ denote the restriction
of $\varphi \in \mathrm{End}(\cA/S)$ to $\cA_s$. The associated group homomorphism
homomorphism
\begin{alignat}1
  \label{eq:specializeend}
  \begin{aligned}
    \mathrm{End}(\cA/S) &\rightarrow \mathrm{End}(\cA_s/\IC(s)),\\
    \varphi&\mapsto \varphi_s
  \end{aligned}  
\end{alignat}
is injective.
As $S$ is smooth, an endomorphism of the generic fiber $\cA_\eta$
extends to an endomorphism of $\cA$ over $S$
by \cite[Proposition I.2.7]{FaltingsChai}. Therefore, (\ref{eq:specializeend}) is
bijective for $s=\eta$.

Observe that any $\varphi\in \mathrm{End}(\cA/S)$ is a
proper morphism. So the image $\varphi(\cA)$ is Zariski closed in
$\cA$. We will consider $\varphi(\cA)$ as a closed subscheme of
$\cA$ with the reduced induced structure. As $\cA$ is reduced,  $\varphi(\cA)$ is the schematic
image of $\varphi$.

The following lemma on endomorphisms of $\cA/S$ relies on a result of
Barroero--Dill and ultimately on the theory of group schemes.

\begin{lemma}
  \label{lem:endabscheme}
  Let $\varphi\in\mathrm{End}(\cA/S)$. Then $\varphi(\cA)$ is an
  abelian subscheme of $\cA$. For all $s\in S(\IC)$ the restriction
  $\varphi_s \colon \cA_s\rightarrow \cB_s$
  is surjective and its kernel has dimension $g-\dim \cB + \dim S$. 
\end{lemma}
\begin{proof}
  Let $B$ be $\varphi(\cA_\eta)$, this is an abelian subvariety of
  $\cA_\eta$ defined over the function field $\IC(\eta)$.
  By \cite[Lemma 2.9]{BarDill} the abelian variety $B$ is the generic fiber of an
  abelian subscheme $\cB\subset \cA$.

  Then $\cA_\eta$ is contain in the closed subset $\varphi^{-1}(\cB)$
  of $\cA$. As $\cA_\eta$ lies dense in $\cA$ we have
  $\varphi(\cA)\subset \cB$, set-theoretically.
  Furthermore, $\varphi$ is proper and its image contains the dense subset $B$ of $\cB$. So 
   $\varphi(\cA)=\cB$ as sets. But $\cA$ and $\cB$ are reduced, so
  $\cB$ is the schematic image of $\varphi$. In particular,
  $\varphi(\cA)$ is an abelian subscheme of $\cA$. 

  For all $s\in S(\IC)$ we have $\varphi(\cA_s) = \cB_s$ and this image has
  dimension $\dim \cB-\dim S$ since $\cB\rightarrow S$ is smooth. The
  lemma follows as the
  kernel of $\varphi_s$ has dimension $\dim \cA_s-\dim \cB_s$.
\end{proof}

We will often treat $\varphi(\cA)$ as an abelian scheme over $S$ and
$\varphi$ as the homomorphism $\cA\rightarrow \varphi(\cA)$.

Let $V$ be an irreducible variety defined over $\IC$. A subset of
$V(\IC)$ is called \textit{meager} in $V$ if it is contained in an at most
countably infinite union of Zariski closed proper subsets of $V$.


We assume that all geometric endomorphisms of the
generic fiber $\cA_\eta$ are defined over the function field $\IC(\eta)$ of $S$.
This condition is met, for example,  if there is an integer $\ell\ge 3$
such that all $\ell$-torsion points of $\cA_{\eta}$ are
$\IC(S)$-rational \cite[Theorem~2.4]{Silverberg:fielddef}.

A point $s\in S(\IC)$ is called \textit{endomorphism generic} for $\cA/S$
if the homomorphism 
\begin{equation}
  \label{eq:endgeneric}
  \mathrm{End}(\cA/S)\otimes\IQ \rightarrow \mathrm{End}(\cA_s/\IC)\otimes\IQ
\end{equation}
induced by (\ref{eq:specializeend})
is surjective. Note that (\ref{eq:endgeneric}) is always injective.
We define
\begin{equation}
  \label{eq:exceptionalsubset}
  S^{\mathrm{exc}} := \bigl\{s \in S(\IC) : s\text{ is not endomorphism
  generic for $\cA/S$}\bigr\}.
\end{equation}

\begin{proposition}
  \label{prop:excep}
  The set $S^{\mathrm{exc}}$ is meager in $S$. 
\end{proposition}
\begin{proof}
  This proposition can be proved using Hodge theory.
  Masser~\cite[Proposition]{Masser:Specialization} gave a 
  proof using an effective Nullstellensatz.
  In this reference one must assume, as we do above, 
  that all geometric endomorphisms of $\cA_\eta$ are
  already defined over $\IC(S)$. 
  As a consequence any endomorphism of $\mathcal{\cA}_s$ for 
  $s$ outside a meager subset
  is the specialization of an endomorphism of the generic
  fiber. Then we use that (\ref{eq:specializeend}) is surjective for $s=\eta$. 
\end{proof}

A \emph{coset} in an abelian variety is the translate of an abelian
subvariety by an arbitrary point. 

\begin{lemma}
  \label{lem:abschquot}
  Let $Y$ be an irreducible closed subvariety of $\cA$ with $\pi(Y)=S$. 
  Assume that there is a Zariski open and dense subset $U\subset
  \pi(Y)$ such that for all $s\in U(\IC)$,
  some irreducible component of  $Y_s$ is a coset in
  $\cA_s$.
  There exists $\varphi\in \mathrm{End}(\cA/S)$
  with the following properties:
  \begin{enumerate}
  \item [(i)] We have   $\dim \varphi(Y) = \dim \pi(Y)$.
  \item[(ii)] 
    For all $s\in S(\IC)$ we have $\dim\ker \varphi_s = \dim Y - \dim\pi(Y)$.
  \end{enumerate}
  Moreover, if $\dim Y = \dim \pi(Y)$, then $\varphi$ is the identity.
\end{lemma}
\begin{proof}
  If $\dim Y = \dim\pi(Y)$ we take $\varphi$ to be the identity, the
  conclusions are all true. 
  Otherwise by generic flatness, we may and do replace $U$ by a  Zariski open and dense
  subset such that
  $Y_s$ is equidimensional of dimension $\dim Y-\dim\pi(Y)$ for all
  $s\in U(\IC)$. 
  

  
  Let $s\in U(\IC)\ssm S^{\mathrm{exc}}$.
  We fix an irreducible component $Z_s$ of $Y_s$
  that is a coset in $\cA_s$, necessarily of dimension $\dim Y
  -\dim\pi(Y)$. Next we pick
  $\varphi_s \in \mathrm{End}({\cA_s}/\IC)$ whose kernel contains
  a translate of the said coset as an irreducible component.
  After multiplying $\varphi_s$ by a
  positive multiple it extends to  an
  endomorphism $\varphi$ of $\cA$ as (\ref{eq:endgeneric}) is
  bijective.
 
  For each $\varphi\in\mathrm{End}(\cA/S)$ we define
  $\Sigma_{\varphi}$ to be the set of points $s\in U(\IC)\ssm
  S^{\mathrm{exc}}$ with $\varphi_s=\varphi$.
  We have
  $S(\IC) = (S\ssm U)(\IC)\cup
  S^{\mathrm{exc}} \cup \bigcup_{\varphi\in \mathrm{End}(\cA/S)} \Sigma_{\varphi}$
  
  The set $  S^{\mathrm{exc}}$ is meager in $S$
  by Proposition~\ref{prop:excep} and so is $ (S\ssm U)(\IC)\cup
    S^{\mathrm{exc}}$.
  As  $\mathrm{End}(\cA/S)$ is at most countably infinite,   
  the Baire Category Theorem implies that there exists
  $\varphi\in\mathrm{End}(\cA/S)$ such that the closure of 
  $\Sigma_{\varphi}$ in $\an{S}$ has non-empty interior.
  In particular, $\Sigma_\varphi$ is Zariski dense in $S$. 

  For all $s\in \Sigma_{\varphi}$ we have $\dim Z_s = \dim Y-\dim\pi(Y)$.
  So $Z = \bigcup_{s\in \Sigma_{\varphi}} Z_s$ lies
  Zariski dense in the irreducible $Y$. 

  For all $s\in \Sigma_{\varphi}$, 
  each $Z_s$ is contained in a fiber of $\varphi|_Y$  by our  choice of $\varphi_s$.
  So $Z_s$, being an irreducible component of
  $Y_s$, is an irreducible component of a fiber of $\varphi|_Y$.
  
  Generically, fibers of $\varphi|_Y$ are equidimensional of  dimension $\dim Y -
  \dim\varphi(Y)$.
  So there exists $s_0\in \Sigma_\varphi$ such that $Z_{s_0}$ meets such a (generic) 
  fiber.
  Then 
  $\dim Y - \dim\varphi(Y) = \dim Z_{s_0}= \dim\ker\varphi_{s_0}$.
  Recall that  $Y_{s_0}$ is equidimensional of dimension $\dim Y -
  \dim \pi(Y)$ and has $Z_{s_0}$ as an irreducible component. We conclude
  $\dim Y - \dim\varphi(Y) = \dim Y -\dim\pi(Y)$ and so 
  $\dim\varphi(Y) =
  \dim\pi(Y)$. This implies (i).
  By Lemma~\ref{lem:endabscheme} all  $\ker\varphi_s$ have
  the same dimension, here equal to $\dim Y
  -\dim\pi(Y)$. This concludes (ii). 
\end{proof}

The exceptional set of an irreducible closed subvariety $Y$ of $\cA$ is defined to be
\begin{equation}
  \label{eq:exceptionalset}
  Y^{\mathrm{exc}}:=     
  \{P \in Y(\IC) :  \text{$P$ is  contained in a proper algebraic
    subgroup of $\cA_{\pi(P)}$}\}.
\end{equation}
If $N\in\IZ$ then $[N]$ denotes the multiplication-by-$N$ morphism
$\cA\rightarrow\cA$. 
\begin{lemma}
  \label{lem:alternative}
  Let $Y$ be an irreducible closed subvariety of $\cA$ and 
  let $S' =\pi(Y)^{\mathrm{reg}}$ denote the regular locus of $\pi(Y)$. We have
  one of the following two alternatives.
  \begin{enumerate}
  \item [(i)]
    Either $Y^{\mathrm{exc}}$ is meager in $Y$,
  \item[(ii)] or every $P\in \pi|_Y^{-1}(S')(\IC)$ lies in a proper algebraic
    subgroup of $\cA_{\pi(P)}$.  In this case,
     $\bigcup_{N\in\IN}[N](Y)$ is not Zariski dense in $\pi^{-1}(\pi(Y))$ and if
     $\eta$ is the generic point of $\pi(Y)$, then $Y_\eta$ lies in a
     proper algebraic subgroup of $\cA_\eta$.
  \end{enumerate}
\end{lemma}
\begin{proof}
 Let  $Y' = Y\cap \cA_{S'}$.
  Suppose $P\in Y'(\IC)$ is in a proper algebraic subgroup of $\cA_{s}$
  with $s=\pi(P)$. Then there exists $\varphi_s\in
  \mathrm{End}(\cA_s/\IC)\ssm\{0\}$ with $\varphi_s(P)=0$. If $s\not\in
  {S'}^{\mathrm{exc}}$, then by definition some positive multiple of
  $\varphi_s$ extends to an element of $\mathrm{End}(\cA_{S'}/S')\ssm\{0\}$.  
  Therefore,
  \begin{equation}
    \label{eq:exceptionalset2}
    Y^{\mathrm{exc}}\subset \pi|_Y^{-1}(\pi(Y)\ssm S')\cup \pi^{-1}({S'}^{\mathrm{exc}}) \cup \bigcup_{\varphi\in
      \mathrm{End}(\cA_{S'}/S')\ssm\{0\}}
    \mathrm{ker}\,\varphi.
  \end{equation}

  By Proposition~\ref{prop:excep} the set
  $\pi|_Y^{-1}({S'}^{\mathrm{exc}})$
  is meager in $Y$. Moreover, $\pi|_Y^{-1}(\pi(Y)\ssm S')$ is Zariski
  closed and proper in $Y$, and hence its complex points form a meager subset of $Y$.
  Moreover, the last union in 
  (\ref{eq:exceptionalset2}) is over an at most
  countably infinite union of proper algebraic subsets of $\cA_{S'}$.
  
  So if we are not in alternative (i), then there exists $\varphi\in
  \mathrm{End}(\cA_{S'}/S')\ssm\{0\}$ with $Y\subset
  \overline{\ker\varphi}$,  the Zariski closure of $\ker\varphi$ in $\cA$.
  Note that $Y' = Y\cap \cA_{S'} \subset \overline{\ker\varphi}\cap
  \cA_{S'} = \ker\varphi$. 
  Say $P\in Y'(\IC)$, then $P\in \ker\varphi_{\pi(P)}$.
  By Lemma~\ref{lem:endabscheme}, $\ker\varphi_{\pi(P)}$ is a proper
  algebraic subgroup of $\cA_{\pi(P)}$. Finally, $[N](Y')\subset\ker\varphi$ for
  all $N\in\IN$. So $\bigcup_{N\in\IN} [N](Y)$
  lies in $\pi|_Y^{-1}(\pi(Y)\ssm S')\cup \ker\varphi$ and is thus not
  Zariski dense in $\pi^{-1}(\pi(Y))$. 
  Finally, the generic fiber of $Y\rightarrow S$ lies in the generic
  fiber of $\ker\varphi\rightarrow S$, the latter is a proper
  algebraic subgroup of $\cA_\eta$. 
\end{proof}

Here is a useful consequence of the previous lemma. 

\begin{lemma}
  \label{lem:meagersuperset}
  Let $Y\subset \cA$ and $X\subset \cA$ be irreducible closed
  subvarieties with $Y\subset X$ and
  $\pi(Y)\cap\pi(X)^{\mathrm{reg}}\not=\emptyset$.
  If $Y^{\mathrm{exc}}$ is meager in $Y$, then
  $X^{\mathrm{exc}}$ is meager in $X$. 
\end{lemma}
\begin{proof}
  If $X^{\mathrm{exc}}$ is not meager in $X$, then by
  Lemma~\ref{lem:alternative} every point $P\in X(\IC)$ with $\pi(P)\in
  \pi(X)^{\mathrm{reg}}$  lies in a proper algebraic subgroup of
  $\cA_{\pi(P)}$. In particular, the complex points of $Y \cap
  \pi^{-1}(\pi(X)^{\mathrm{reg}})$ lie in  $Y^{\mathrm{exc}}$.
  The hypothesis implies that $Y^{\mathrm{exc}}$ contains a non-empty open subset
  of $Y^{\mathrm{an}}$. So 
  $Y^{\mathrm{exc}}$ is not meager in $Y$
  by   the Baire Category Theorem.
\end{proof}

\section{Bi-algebraic Subvarieties and the University Family of Abelian Varieties}

Ullmo and Yafaev~\cite{UYweaklyspecial} characterized the bi-algebraic
subvarieties of (pure) Shimura varieties: they are precisely the
weakly special subvarieties, \textit{i.e.}, the geodesic subvarieties
studied by Moonen~\cite{Moonen}.
For a definition of bi-algebraic subgroup we refer to
$\mathsection$\ref{sec:bialgebraic} below.
The first-named
author~\cite[$\mathsection$3]{Gao:SNS17} gave a complete
characterization of the bi-algebraic subvarieties of $\mathfrak{A}_g$,
based on \cite[$\mathsection$8]{GaoCrelle}. Below in
Proposition~\ref{prop:monodromyaction} we follow the approach
presented in these references but minimize the language of mixed
Shimura varieties. Our main tool is Andr\'e's normality
theorem~\cite{Andre92} for variations of mixed Hodge structures.

\subsection{The Mumford--Tate Group}
\label{sec:MT}

Let $g\ge 1$ and let $\pi\colon \mathfrak{A}_g\rightarrow\IA_g$ be the
universal family of principally polarized $g$-dimensional abelian
varieties with level-$\ell$-structure for some $\ell\ge 3$. Then
$\mathfrak{A}_g$ and $\IA_g$ are geometrically irreducible, smooth
quasi-projective varieties defined over a number field which we assume
is a subfield $\IC$. We consider all varieties as defined over a
subfield of $\IC$, sometimes executing a base change to $\IC$ without
mention.

Let $\mathfrak{H}_g$ denote Siegel's upper half space, \textit{i.e.}, the
symmetric matrices in $\mathrm{Mat}_{g\times g}(\IC)$ with positive definite
imaginary part. 
By abuse of notation we write
\begin{equation*}
  \mathrm{unif}\colon \mathfrak{H}_g
  \rightarrow\IA_g^{\mathrm{an}}\quad\text{and}\quad
  \mathrm{unif}\colon \IC^g\times \mathfrak{H}_g\rightarrow\mathfrak{A}_g^{\mathrm{an}}
\end{equation*}
for both holomorphic uniformizing maps.
Recall that $\mathrm{Sp}_{2g}(\IR)$, the group of real points of the
symplectic group,
acts on $\mathfrak{H}_g$.

We identify $\IR^g\times\IR^g\times \mathfrak{H}_g$
 with
$\IC^g\times \mathfrak{H}_g$ via the natural semi-algebraic bijection
\begin{equation}
  \label{eq:complexstructure}
(\tau,u,v)  \leftrightarrow
(\tau,z)\quad\text{where $z = \tau u + v$.}
\end{equation}
In the former coordinates, the
corresponding uniformizing map $\mathfrak{H}_g
\times\IR^{2g}\rightarrow\mathfrak{A}_g^{\mathrm{an}}$ is real
analytic.

Let $s\in \mathfrak{A}_g(\IC)$ and fix 
$\tau\in\mathfrak{H}_g$ in its preimage under the uniformizing map, \textit{i.e}, $s =
\mathrm{unif}(\tau)$. Let $1_g$ denote the $g\times g$ unit matrix,
then the columns of $(\tau,1_g)$ are an $\IR$-basis of $\IC^g$ and
$\mathfrak{A}_{g,s}^{\mathrm{an}} \cong \IC^g/(\tau\IZ^g +
\IZ^g)$.
The period lattice basis $(\tau,1_g)$ allows us to identify
$H_1(\mathfrak{A}_{g,s}^{\mathrm{an}},\IZ)$ with $\IZ^{g}\times\IZ^g$
and $H_1(\mathfrak{A}_{g,s}^{\mathrm{an}},\IR)$ with
$\IR^g\times\IR^g$.

We briefly recall the monodromy action of
$\pi_1(\IA_g^{\mathrm{an}},s)$, the (topological) fundamental group of
$\IA_g^{\mathrm{an}}$ based at $s$, on singular homology
$H_1(\mathfrak{A}_{g,s}^{\mathrm{an}},\IZ)$.

Suppose $[\gamma] \in \pi_1(\IA_g^{\mathrm{an}},s)$ is represented by
a loop $\gamma$ in $\IA_{g}^{\mathrm{an}}$ based at $s$. Then a lift
$\tilde\gamma$ of $\gamma$ to $\mathfrak{H}_g$ starting at $\tau$ ends
at $M \tau\in \mathfrak{H}_g$ for some $M=\left(
  \begin{array}{cc} a&b \\c & d
  \end{array}\right)\in \mathrm{Sp}_{2g}(\IZ)$. Then $M\tau$ is the
period matrix of the abelian variety $\IC^g/(M\tau\IZ^g+\IZ^g)$ which
is isomorphic to $\IC^g / (\tau\IZ^g+\IZ^g)$. To describe this
isomorphism we need the identity
\begin{equation}
  \label{eq:analyticiso}
  I(M,\tau)^\trans (M\tau,1_g) = (\tau,1_g) M^\trans,  
\end{equation}
where $I(M,\tau)=c\tau+d$, note the transpose and
see \cite[$\mathsection$8.1 and Remark~8.1.4]{CAV}
for a discussion. We rearrange this equation.
The map
\begin{alignat}1
  \label{eq:monodromyiso}
\tau u + v\mapsto (I(M,\tau)^\trans)^{-1}(\tau u+v)=
(M\tau,1_g)(M^\trans)^{-1}
\left(
  \begin{array}{c}
    u \\ v
  \end{array}
\right),
\end{alignat}
here $u,v\in\IR^g$ are column vectors, induces the isomorphism $\IC^g /
(\tau\IZ^g+\IZ^g)\rightarrow\IC^g/(M\tau\IZ^g+\IZ^g)$.

By (\ref{eq:monodromyiso}), the monodromy representation 
expressed in these coordinates is given by
\begin{equation}
  \label{eq:monodromyrep}
  \begin{aligned}
    \rho \colon \pi_1(\IA_g^{\mathrm{an}},s)&\rightarrow   \mathrm{Sp}_{2g}(\IZ)\\
    [\gamma]&\mapsto (M^\trans)^{-1}.
  \end{aligned}
\end{equation}


Next we recall the definition of the Mumford--Tate group in our context.


We continue to assume $\tau\in\mathfrak{H}_g$.
Choose any $M \in \mathrm{Sp}_{2g}(\IR)$
with $\tau = M(\sqrt{-1}\cdot 1_g)$; such an $M$ exists as
$\mathrm{Sp}_{2g}(\IR)$ acts transitively on $\mathfrak{H}_g$. We set
\begin{equation}
  \label{eq:formulaJ}
  J_\tau = (M^{\trans})^{-1} \Omega
  M^{\trans}\quad\text{where}\quad \Omega=\left(
    \begin{array}{cc}
      0 & 1_g \\ -1_g & 0
    \end{array}\right). 
\end{equation}

We claim that $J_\tau$ is independent
of the choice of $M$; it depends only on $\tau$. Indeed,
if $M'$ is a further element of $\mathrm{Sp}_{2g}(\IR)$ with $\tau =
M' (\sqrt{-1} \cdot 1_g)$, then $M = M' N$ where $N\in
\mathrm{Sp}_{2g}(\IR)$ stabilizes $\sqrt{-1}\cdot 1_g$. So $N$ is of
the form $\left(
  \begin{array}{cc} a & b \\ -b & a
  \end{array}\right)$ where $a,b\in \mat{g}{\IR}$. This implies
$(N^\trans)^{-1} \Omega N^\trans = \Omega$ and so $({M'}^\trans)^{-1} \Omega
{M'}^\trans=J_\tau$ on substituting $M'=MN^{-1}$.

Say $x,y\in\IR$, then
\begin{equation*}
  (x1_{2g}+yJ_\tau)^\trans \Omega (x1_{2g}+yJ_\tau) = x^2 \Omega +y^2 J_\tau^\trans \Omega J_\tau +
  xy(\Omega J_\tau + J_\tau^\trans\Omega). 
\end{equation*}
The group $\mathrm{Sp}_{2g}(\IR)$ contains $\Omega$ and is mapped to
itself by matrix transposition. Hence $J_\tau\in \mathrm{Sp}_{2g}(\IR)$.
Moreover, $J_\tau^2 = -1_{2g}$.
So $J_\tau^\trans \Omega J_\tau=\Omega$
and $J_\tau^\trans \Omega = \Omega J_\tau^{-1} = -\Omega J_\tau$. We conclude $h_\tau(z) = x 1_{2g}+y J_\tau \in
\mathrm{GSp}_{2g}(\IR)$ for all
$z=x+\sqrt{-1}y\in\IC^\times=\IC\ssm\{0\}$ where $x,y\in\IR$.
Moreover, 
\begin{equation*}
  h_\tau \colon \IC^\times \rightarrow\mathrm{GSp}_{2g}(\IR)
\end{equation*}
is a group homomorphism.

By (\ref{eq:formulaJ}) we have $J_\tau^\trans \tau = \tau$. 
Below we use the well-known identity $I(MM',\tau)=I(M,M'\tau)I(M',\tau)$ for all $M,M'\in
\mathrm{Sp}_{2g}(\IR)$ and all $\tau\in\mathfrak{H}_g$.
We apply  (\ref{eq:analyticiso}) to $J_\tau^\trans$
where $J_\tau = (M^\trans)^{-1}\Omega M^\trans$ and $\tau = M(\sqrt{-1}\cdot 1_g)$
and compute
\begin{alignat*}1
  (\tau,1_g)J_\tau &= I(J_\tau^\trans,\tau)^\trans (\tau,1_g) \\
              &= I(-M\Omega M^{-1},\tau)^\trans(\tau,1_g) \\
              &=  \left( I(-M\Omega,M^{-1}\tau) I(M^{-1},\tau)\right)^\trans  (\tau,1_g) \\
  &=-\left(I(M\Omega,\sqrt{-1} \cdot 1_g)I(M^{-1},\tau)\right)^\trans (\tau,1_g)\\
  &=-\left(I(M,\Omega(\sqrt{-1} \cdot 1_g))I(\Omega,\sqrt{-1}\cdot
    1_g)I(M^{-1},\tau)\right)^\trans(\tau,1_g).
\end{alignat*}
Next we use $\Omega(\sqrt{-1} \cdot 1_g) = \sqrt{-1}\cdot 1_g$. Hence
\begin{equation}
  \label{eq:Jtauisi}
  \begin{aligned}
    (\tau,1_g)J_\tau&=\sqrt{-1}\left(I(M,\sqrt{-1} \cdot 1_g)I(M^{-1},\tau)\right)^\trans(\tau,1_g)\\
    &= \sqrt{-1}\left(I(M,M^{-1}\tau)I(M^{-1},\tau)\right)^\trans(\tau,1_g)\\
    &= \sqrt{-1}I(1_{2g},\tau)^\trans(\tau,1_g)\\
    &=\sqrt{-1}(\tau,1_g).
  \end{aligned}
\end{equation}
So $J_\tau$ represents multiplication by
$\sqrt{-1}$ in the real coordinates determined by the $\IR$-basis $(\tau,1_g)$ of
$\IC^g$.

Let $s\in\IA_g(\IC)$ lie below $\tau\in\mathfrak{H}_g$. 
The \emph{Mumford--Tate group} $\mathrm{MT}(\mathfrak{A}_{g,s})$ of  $\mathfrak{A}_{g,s}$ is the
smallest algebraic subgroup of $\mathrm{GSp}_{2g,\IQ}$
whose group of  $\IR$-points contains $h_\tau(\IC^\times)$. 
As $J_\tau = h_\tau(\sqrt{-1})$ we certainly have
$J_\tau\in \mathrm{MT}(\mathfrak{A}_{g,s})(\IR)$. 


\subsection{Bi-algebraic Subvarieties}
\label{sec:bialgebraic}

We keep the conventions introduced in the beginning of
$\mathsection$\ref{sec:MT}. 
An irreducible closed subvariety $Y\subset\mathfrak{A}_g$ is called
\emph{bi-algebraic}, if some (or equivalently any) complex analytic irreducible component
of $\mathrm{unif}^{-1}(Y^{\mathrm{an}})$ equals an irreducible
component of $\tilde Y(\IC)\cap
(\IC^g\times \mathfrak{H}_g)$ for an algebraic subset
$\tilde Y\subset \IG_{\mathrm{a},\IC}^g\times \mathrm{Mat}_{g\times
  g,\IC}$.
All irreducible components of the intersection of $2$ bi-algebraic subvarieties of $\mathfrak{A}_g$
are bi-algebraic. So any
 irreducible closed subvariety $Y$ of $\mathfrak{A}_g$ is contained
in a bi-algebraic subvariety $Y^{\mathrm{biZar}}$ of
$\mathfrak{A}_g$ that is minimal with respect to inclusion.

Bi-algebraic subvarieties of $\IA_g$ are defined in a similar manner.
By a theorem of Ullmo--Yafaev~\cite[Theorem 1.2]{UYweaklyspecial}, the bi-algebraic subvarieties
of $\IA_g$ are precisely the weakly special subvarieties of $\IA_g$;
here we consider
$\IA_g$ as a Shimura variety. 
For any irreducible closed subvariety $Y$ of $\IA_g$, we use 
$Y^{\mathrm{biZar}}$ to denote the minimal bi-algebraic subvariety
containing $Y$.

\begin{lemma}
  \label{lem:YbiZarfiber}
  Let $Y\subset \mathfrak{A}_{g}$
  be an irreducible closed
  subvariety that is bi-algebraic and
  let $\eta\in \pi(Y)$ be the  generic point.
  For all $P\in Y(\IC)$, each irreducible component of 
  $Y_{\pi(P)}$ is a coset in
  $\mathfrak{A}_{g,\pi(P)}$ of
  dimension at least $\dim Y - \dim \pi(Y)$.
\end{lemma}
\begin{proof}   
  Let $P\in Y^{\mathrm{biZar}}(\IC)$ and let  $C$ be an irreducible component of
  $Y^{\mathrm{biZar}}_{\pi(P)}$.  By \cite[Corollary 14.116
  and Remark 14.117]{GoertzWedhorn} we have
  $\dim C \ge \dim Y - \dim \pi(Y)$.

  The irreducible component $C$ of $Y_{\pi(P)}$ is a bi-algebraic
  subset of the abelian variety $\mathfrak{A}_{g,\pi(P)}$. The lemma
  follows as by
  \cite[Proposition~5.1]{UYweaklyspecial}, $C$ is a coset in the ambient
  abelian variety.
\end{proof}

We now come to a structural  result of bi-algebraic subsets. We refer
to the first-named author's more comprehensive result in \cite{Gao:SNS17} (the statement of \cite[Proposition~3.3]{Gao:SNS17} contains a mistake; for a correct version see \cite[Proposition~5.3]{GaoBettiRank}) using the language of mixed Shimura varieties.

\begin{proposition}
  \label{prop:monodromyaction}
  Let $Y\subset\mathfrak{A}_g$ be a bi-algebraic subvariety with
  $Y^{\mathrm{exc}}$ meager in $Y$.
  \begin{enumerate}
  \item [(i)]
    There is a vector space $W\subset\IR^{2g}$ defined over $\IQ$ with
    $\dim W = 2(\dim Y - \dim\pi(Y))$ with the following property. 
    For 
    all $s=\mathrm{unif}(\tau)\in \pi(Y)(\IC)$, with $\tau\in\mathfrak{H}_g$, the fiber $Y_s$ is a finite union of
    translates of  $\mathrm{unif}(\{\tau\}\times W) \subset
    \mathfrak{A}_{g,s}^{\mathrm{an}}$ which is an abelian variety $C_s$
  \item[(ii)]
    The quotient abelian varieties $\mathfrak{A}_{g,s}/C_s$ are pairwise isomorphic abelian
    varieties for all $s\in\pi(Y)(\IC)$.    
  \end{enumerate}
\end{proposition}
Each $\tau \in \mathfrak{H}_g$ 
endows $\IR^{2g}$ with the structure of a $\IC$-vector space,
multiplication by $\sqrt{-1}$ is represented by $J_\tau$ from (\ref{eq:formulaJ}). The
subspace $W\subset\IR^{2g}$ from part (i) is a $\IC$-vector space for
all $\tau$ in question. The image  $\mathrm{unif}(\{\tau\}\times W)$
is an abelian subvariety of $\mathfrak{A}_{g,s}$ of dimension $\dim Y
- \dim \pi(Y)$. In particular, $Y\rightarrow\pi(Y)$ is
equidimensional.

\begin{proof}
  If $\pi(Y)$ is a point, say $s\in \IA_g(\IC)$, then $Y_s$ is a coset
  in $\mathfrak{A}_{g,s}$ by Lemma~\ref{lem:YbiZarfiber}. The
  proposition holds in this case. 

  We will now assume $\dim\pi(Y)\ge 1$. 
  We identify $\IR^{2g}\times\mathfrak{H}_g$ with the universal
  covering of $\mathfrak{A}_{g}(\IC)$; sometimes alluding to the
  complex structure induced by (\ref{eq:complexstructure}). The
  fundamental group of $\mathfrak{A}_g^{\mathrm{an}}$ based at some
  point $P$  is a
  subgroup of $\IZ^{2g}\rtimes \mathrm{Sp}_{2g}(\IZ)$. The element
  $(M,\omega)$ acts by
  \begin{equation*}
    \left(M\tau, M*u+\omega
    \right)
  \end{equation*}
  on $(\tau,u)$; where $M*u = (M^{\trans})^{-1}u$.

  Recall that the  ambient variety $\mathfrak{A}_g$ is
  quasi-projective and so is $Y$. By Bertini's Theorem a general
  linear space of codimension $\dim Y - 1$
  intersected with $Y^{\mathrm{reg}}$ is a smooth, irreducible curve
  $\mathbf{x}$ that is quasi-finite over $\pi(\mathbf{x})$.  
  A suitable version of Lefschetz's Theorem for the topological
  fundamental group we may also assume that
  the homomorphism
  \begin{equation}
    \label{eq:lefschetz}
    \pi_1(\mathbf{x}^{\mathrm{an}},P)\rightarrow \pi_1(Y^{\mathrm{reg,an}},P)
  \end{equation}
  induced by the inclusion $\mathbf{x}\rightarrow Y^{\mathrm{reg,an}}$
  is surjective for all $P\in \mathbf{x}(\IC)$; see \cite[Lemme
  1.4]{Deligne:bourbaki81}.
  We may fix  $P$ in very general position. For
  example, $P$ is not contained in a proper algebraic subgroup of
  $\mathfrak{A}_{g,s}$ for $s=\pi(P)$. If we replace $\mathbf{x}$ by a
  Zariski open and dense subset, the image of the induced homomorphism
  has finite index in $\pi_1(Y^{\mathrm{reg,an}},P)$, \cite[Lemme
  4.4.17]{Deligne:Hodge2}. This suffices for us. So we
  may assume that $\pi|_{\mathbf{x}} \colon \mathbf{x}\rightarrow \pi(\mathbf{x})$
  is finite and \'etale.

  Let $\Gamma$ denote the image of $\pi_1(\mathbf{x}^{\mathrm{an}},P)$
  in $\IZ^{2g}\rtimes \mathrm{Sp}_{2g}(\IZ)$.
  Let $\mathrm{Mon}(\mathbf{x})$ be the neutral component of the Zariski closure
  of $\Gamma$ in $\IG_{\mathrm{a},\IQ}^{2g}\rtimes \mathrm{Sp}_{2g,\IQ}$
  and let $\mathrm{Mon}(Y^{\mathrm{reg}})$ be the neutral component of
  the Zariski closure of the image of
  $\pi_1(Y^{\mathrm{reg,an}},P)$.  
  We call  $\mathrm{Mon}(\mathbf{x})$ the \emph{connected algebraic monodromy group} of
  $\mathbf{x}$. By the surjectivity of (\ref{eq:lefschetz}) and the
  discussion below we have
  \begin{equation}
    \label{eq:MonxMonY}
    \mathrm{Mon}(\mathbf{x})=\mathrm{Mon}(Y^{\mathrm{reg}}).
  \end{equation}

  By Lemma~\ref{lem:YbiZarfiber} 
  we have $P\in C(\IC)$ where $C$ is an irreducible component of  
  $ Y_s$ and  a coset in $\mathfrak{A}_{g,s}$ with $\dim C\ge \dim Y -
  \dim\pi(Y)$. Now $C\cap Y^{\mathrm{reg}}$ is Zariski dense and open
  in $C$. So the image of  $\pi_1(C^{\mathrm{an}}\cap Y^{\mathrm{reg,an}},P)$ in
  $\pi_1(C^{\mathrm{an}},P)$, induced
  by inclusion, has
  finite index. But
  $\pi_1(C^{\mathrm{an}},P)$  can be identified with a
  subgroup of $\IZ^{2g}\cong H_1(\mathfrak{A}_{g,s}^{\mathrm{an}},\IZ)$ of rank $2\dim C\ge 2 (\dim Y-\dim\pi(Y))$.

  The kernel of the  projection
  $\mathrm{pr}\colon \IG_{\mathrm{a},\IQ}^{2g}\rtimes \mathrm{Sp}_{2g,\IQ}\rightarrow
  \mathrm{Sp}_{2g,\IQ}$
  restricted to 
  $\mathrm{Mon}(Y^{\mathrm{reg}})$
  is an algebraic subgroup of  $\IG_{\mathrm{a},\IQ}^{2g}\times \{1_{2g}\}$. So it is
  $\{1_{2g}\}\times W$ with $W$ a linear subspace of $\IG_{\mathrm{a},\IQ}^{2g}$. By
  the previous paragraph and by (\ref{eq:MonxMonY}) we have
  \begin{equation}
    \label{eq:dimWlb}
    \dim W \ge 2 (\dim Y - \dim\pi(Y)).
  \end{equation}

  Let $Z = \pi(Y)^{\mathrm{reg}}$.
  Then there is a natural representation 
  $\pi_1(Z^{\mathrm{an}},s)\rightarrow\mathrm{Sp}_{2g}(\IZ)$.
  The connected algebraic monodromy group  $\mathrm{Mon}(Z)\subset\mathrm{Sp}_{2g,\IQ}$ is
  the neutral component of the Zariski closure of the image
  of $\pi_1(Z^{\mathrm{an}},s)$.

  Note that $\mathrm{pr}(\mathrm{Mon}(Y^{\mathrm{reg}})) =
  \mathrm{Mon}(Z)$  by \cite[Lemme
  4.4.17]{Deligne:Hodge2}.

  Let $M \in \mathrm{Mon}(Z)(\IC)$.
  The preimage $\mathrm{pr}|_{\mathrm{Mon}(Y^{\mathrm{reg}})}^{-1}(M)$ is $(M,\psi(M)+W)$ where $\psi(M)$
  is a unique complex point of
  $\IG_{\mathrm{a},\IQ}^{2g}/W$.  
  For all $M,M'\in \mathrm{Mon}(Z)(\IC)$ we have
  \begin{equation*}
    (M,\psi(M)+W)(M',\psi(M')+W) =
    (MM',\psi(M)+M*\psi(M')+W) = (MM',\psi(MM')+W).
  \end{equation*}
  So $\psi \colon \mathrm{Mon}(Z)\rightarrow\IC^{2g}/W(\IC)$ is a cocycle.
  It must be a coboundary as $\mathrm{Mon}(Z)$ is semi-simple or trivial,
  by work of Deligne
  \cite[Corollaire~4.2.9(a)]{Deligne:Hodge2}.
  Hence there exists 
  $v_0\in \IC^{2g}$ with
  \begin{equation*}
    \psi(M) = M*v_0-v_0+W\quad\text{for all}\quad M\in \mathrm{Mon}(Z)(\IC).    
  \end{equation*}

  Let $\tilde Y$ be an algebraic subset of
  $\IG_{\mathrm{a},\IC}^g\times\mathrm{Mat}_{g\times
    g,\IC}$ such that 
  the preimage of $Y(\IC)$ in $\IC^{g}\times\mathfrak{H}_g$ and 
  $\tilde Y(\IC) \cap
  (\IC^g\times \mathfrak{H}_g)$ have a common complex analytic
  irreducible component, say $\tilde Y_0$. We have $\dim \tilde Y_0 =
  \dim Y$.

  Suppose $\tau_0\in\mathfrak{H}_g$ lies above $s$ and
  $(\tau_0, u_0)\in \tilde Y_0$ lies above $P$. 
  We return to real coordinates, so $u_0\in\IR^{2g}$. 

  An element $[\gamma]\in \Gamma$ is represented by a loop $\gamma$ in
  $\mathbf{x}^{\mathrm{an}}$ based at $P$. We lift $\gamma$ to an arc in
  $\IR^{2g}\times\mathfrak{H}_g$ starting at $(\tau_0,u_0) \in \tilde
  Y(\IC)$. 
  In particular, the end point of the lift lies in $\tilde Y_0$ and
  equals $(M,u)(\tau_0,u_0)$ with
  $(M,u)=[\gamma]\in\Gamma$.
  For the orbit of $(\tau_0,u_0)$ under $\Gamma$ we have
  \begin{equation*}
    \Gamma \cdot (\tau_0,u_0)\subset \tilde Y_0\subset
    \tilde Y(\IC).
  \end{equation*}




  By definition $\tilde Y(\IC)$ is an algebraic subset of
  $\mathrm{Mat}_{g\times g}(\IC)\times\IC^{g}$. As $P$ is a regular
  point of $Y$ we have that $(\tau_0,u_0)$ is a regular point of
  $\tilde Y_0$. So
  \scriptsize
  \begin{equation}
    \label{eq:algebraicorbit}
    \mathrm{Mon}(Y^{\mathrm{reg}})(\IR)^+ \cdot (\tau_0,u_0)=
    \mathrm{Mon}(\mathbf{x})(\IR)^+ \cdot (\tau_0,u_0) = \{(M,u)(\tau_0,u_0) :
    (M,u)\in     \mathrm{Mon}(\mathbf{x})(\IR)^+ \} \subset \tilde
    Y_0; 
  \end{equation}  
  \normalsize
  the superscript $+$ signals taking the neutral component in the
  Euclidean topology.
  
  In addition to the connected algebraic monodromy group, we have the
  corresponding Mumford--Tate group. 

  First, recall that $s\in \pi(Y)(\IC)$ determines a principally polarized
  abelian variety $\mathfrak{A}_{g,s}$ defined over $\IC$. We consider
  its Mumford--Tate group $\mathrm{MT}(\mathfrak{A}_{g,s})$ coming
  from the corresponding weight $-1$ pure Hodge structure; it is a
  reductive algebraic group. Moreover,
  $\mathrm{MT}(\mathfrak{A}_{g,s})$ is naturally an algebraic subgroup
  of $\mathrm{GSp}_{2g,\IQ}$.

  Second, after a finite and \'etale base change, which is harmless
  for our investigations, $\mathbf{x}$ becomes a section of an abelian
  scheme. It is a good, smooth one-motive (of rank $\le 1$) in the sense
  of Deligne; see \cite[$\mathsection$4 and Lemma~5]{Andre92}. Attached
  to $\mathbf{x}$ is a variation of mixed Hodge structures. 
  Restricted to each point of $\mathbf{x}$ we obtain a mixed Hodge
  structure. By \cite[$\mathsection$4 and Lemma~5]{Andre92} the mixed Hodge structure thus obtained
  has the same the Mumford--Tate group  for all sufficiently 
  general points in  $\mathbf{x}$. We denote this group by
  $\mathrm{MT}(\mathbf{x})$. We may assume $P$ to be such a very general
  point. Then $\mathrm{MT}(\mathbf{x})$ is naturally an algebraic
  subgroup of $\IG_{\mathrm{a}}^{2g}\rtimes \mathrm{GSp}_{2g,\IQ}$.

  We also write $\mathrm{pr}$ for the projection
  $\IG_{\mathrm{a},\IQ}^{2g}\rtimes \mathrm{GSp}_{2g,\IQ}
  \rightarrow \mathrm{GSp}_{2g,\IQ}$. By
\cite[Lemma~2(c)]{Andre92} we have surjectivity
  $\mathrm{pr}(\mathrm{MT}(\mathbf{x})) =
  \mathrm{MT}(\mathfrak{A}_{g,s})$.



  Andr\'e \cite[Theorem~1]{Andre92} proves that
  $\mathrm{Mon}(\mathbf{x})$ is a normal subgroup of
  $\mathrm{MT}(\mathbf{x})$ as $P$ is in very general position. 
  We do not require the statement that $\mathrm{Mon}(\mathbf{x})$ is
  in fact a normal subgroup of the derived Mumford--Tate group.

  Before moving on, we make the following remark. The second remark on \cite[page~11]{Andre92} suggests that we
  could work
  directly with $Y^{\mathrm{reg}}$, \textit{i.e.}, without bypassing to
  $\mathbf{x}$.

  Let $M\in \mathrm{Mon}(Z)(\IC)$ be arbitrary. Then
  $(M,M*v_0-v_0+W)\subset \mathrm{Mon}(\mathbf{x})(\IC)$. 
  Suppose $(h,v)\in \mathrm{MT}(\mathbf{x})(\IC)$. By Andr\'e's
  Theorem, we have
  \begin{equation*}
    (h,v) (M,M*v_0-v_0+W) (h,v)^{-1} \subset
    \mathrm{Mon}(\mathbf{x})(\IC),
  \end{equation*}
  so
  \begin{equation*}
    (hMh^{-1},v+hM*v_0-h*v_0 -hMh^{-1}*v+h*W)\subset
    \mathrm{Mon}(\mathbf{x})(\IC).
  \end{equation*}
  Recall (\ref{eq:MonxMonY}). The second coordinate equals $\psi(M)= hMh^{-1}*v_0-v_0+ W$, so 
  \begin{equation*}
    v+hM*v_0-h*v_0-hMh^{-1}*v + h*W= hMh^{-1}*v_0-v_0 + W.
  \end{equation*}
  We draw two conclusions.
  
  First, we have $h*W=W$. As the projection
  $\mathrm{MT}(\mathbf{x})\rightarrow\mathrm{MT}(\mathfrak{A}_{g,s})$ is
  surjective, it follows that $\mathrm{MT}(\mathfrak{A}_{g,s})$ acts on
  $W$.
  The reductive group $\mathrm{MT}(\mathfrak{A}_{g,s})$  also acts on a
  linear subspace $W^{\perp} \subset \IG_{\mathrm{a},\IQ}^{2g}$ with
  $\IG_{\mathrm{a},\IQ}^{2g}= W\oplus W^{\perp}$.
  
  Second, 
  and putting $h=1$,  we get 
  \begin{equation}
    \label{eq:monactionW}
    M*v-v \in W(\IC)
    \quad\text{for all}\quad M\in \mathrm{Mon}(Z)(\IC) \text{ and all
    }
    (1,v) \in \mathrm{MT}(\mathbf{x})(\IC). 
  \end{equation}
  
  Now let us compute the kernel of
  $\mathrm{MT}(\mathbf{x})\rightarrow
  \mathrm{MT}(\mathfrak{A}_{g,s})$ using  \cite[Proposition 1]{Andre92}. In its notation we set
  $H=G=\mathrm{MT}(\mathbf{x})$ and claim $E' = \mathfrak{A}_{g,s}$.
  Indeed, $P$ is not contained in a proper algebraic subgroup of
  $\mathfrak{A}_{g,s}$ by hypothesis. So the cyclic subgroup it
  generates is Zariski dense in $\mathfrak{A}_{g,s}$. The said
  proposition then implies that the kernel equals $\IG_{\mathrm{a},\IQ}^{2g}$.
  So (\ref{eq:monactionW}) holds for all $v\in\IC^{2g}$. 

  In particular,

  $M*v-v\in W(\IC)$ for all $M\in \mathrm{Mon}(Z)(\IC)$
  and all $v\in W^\perp(\IC)$. As $\mathrm{Mon}(Z)$ is an algebraic
  subgroup of $\mathrm{MT}(\mathfrak{A}_{g,s})$ it also acts on
  $W^\perp$. As $W(\IC)\cap W^{\perp}(\IC)=0$ we conclude that
  $\mathrm{Mon}(Z)$ acts trivially on $W^{\perp}$. So $W^{\perp}(\IQ)$
  is contained in the fixed part of the monodromy action on
  $H_1(\mathfrak{A}_{g,s},\IQ)$.

  Moreover, we write $v_0 = v'_0+v''_0$ such that $v'_0\in W(\IR)$
  and $v''_0\in W^\perp(\IR)$. Then
  $M*v_0 = M*v'_0+M*v''_0 \in v''_0+ W=v_0 + W$ and
  \begin{equation*}
    \psi(M) = M*v'_0+M*v''_0 - v'_0 -v''_0 + W = M*v'_0 - v'_0 + W = W
  \end{equation*}
  for all $M\in \mathrm{Mon}(Z)(\IC)$.

  We summarize these last arguments and (\ref{eq:MonxMonY}) by stating that the connected algebraic monodromy group satisfies
  \begin{equation*}
     \mathrm{Mon}(Y^{\mathrm{reg}})=
    \mathrm{Mon}(\mathbf{x})
    =W\rtimes \mathrm{Mon}(Z).
  \end{equation*}
  By (\ref{eq:algebraicorbit}) we have
  \scriptsize
  \begin{alignat*}1
    (\mathrm{Mon}(Z)(\IR)^+\cdot\tau_0)\times (v_0+W(\IR))=\{(M\tau_0,v_0+w) : M\in \mathrm{Mon}(Z)(\IR)^+, w\in W(\IR)\}
    \subset \tilde Y_0 
  \end{alignat*}
  \normalsize
  
  By Moonen's work on weakly special subvarieties the orbit
  $\mathrm{Mon}(Z)(\IR)^+\cdot \tau_0$ maps onto $\pi(Y)$ under the
  uniformizing map $\mathfrak{H}_g
  \rightarrow\mathbb{A}_g^{\mathrm{an}}$, see \cite[$\mathsection$3 and
  Proposition~3.7]{Moonen}. Generically, the fiber of
  $Y\rightarrow\pi(Y)$ has dimension $\dim Y-\dim\pi(Y)$, which is $\le
  \frac 12 \dim_{\IR} W(\IR)$ by (\ref{eq:dimWlb}). Hence
  $\dim_{\IR}\tilde{Y}_0 \le \dim_\IR(
\mathrm{Mon}(Z)(\IR)^+\cdot\tau_0)\times (v_0+W(\IR))$.

  We now show
  $(\mathrm{Mon}(Z)(\IR)^+\cdot\tau_0)\times (v_0+W(\IR))=\tilde Y_0$.
  Indeed, note that the left-hand side is closed in
  $\tilde Y_0$. 
  Let $T$ denote the singular points of the complex analytic
  space
  $\tilde Y_0$.
  As $\tilde Y_0$ is irreducible, 
 $\tilde Y_0 \ssm T$ is a connected complex manifold.
 Moreover, $(\mathrm{Mon}(Z)(\IR)^+\cdot\tau_0)\times (v_0+W(\IR))\ssm
  T$ is a topological (real) manifold of dimension ${2\dim \tilde Y_0}$ contained
  in $\tilde Y_0\ssm T$. So it is open
  in $\tilde Y_0\ssm T$ by invariance of dimension. But it is also closed in $\tilde Y_0\ssm T$.
  So
  $(\mathrm{Mon}(Z)(\IR)^+\cdot\tau_0)\times (v_0+W(\IR))\ssm T=\tilde
  Y_0\ssm T$. The claim follows as $\tilde Y_0\ssm T$ is dense in
  $\tilde Y_0$.

  In particular, $(\mathrm{Mon}(Z)(\IR)^+\cdot\tau_0)\times
  (v_0+W(\IR))$ is complex analytic. Thus for all $\tau \in
  \mathfrak{H}_g$ with $\mathrm{unif}(\tau) \in \pi(Y)(\IC)$, $W(\IR)$
  is a complex subspace for the complex structure on $\IR^{2g}$ endowed
  by $\tau$. Moreover, (\ref{eq:dimWlb}) is an equality.
  
  This concludes (i) since $W$ is an algebraic subgroup of
  $\IG_{\mathrm{a},\IQ}^{2g}$. Part (ii) follows from
  \cite[Corollaire~4.1.2]{Deligne:Hodge2} because
  $W^{\perp}(\IQ)$ is contained in the fixed part of the monodromy
  action on $H_1(\mathfrak{A}_{g,s},\IQ)$.
\end{proof}

We end this section with a sufficient criterion for the meagerness of the bi-algebraic
closure of a variety.


\begin{lemma}
  \label{lem:YcapYbizarreg}
  Let $Z$ be an irreducible closed subvariety of $\IA_g$, 
  then $Z\cap Z^{\mathrm{biZar,reg}}\not=\emptyset$.
  Let $Y$ be an irreducible closed subvariety of $\mathfrak{A}_g$.
  Then $\pi(Y)\cap
  \pi(Y^{\mathrm{biZar}})^{\mathrm{reg}}\not=\emptyset$.
  If $Y^{\mathrm{exc}}$ is meager in $Y$, then
  $Y^{\mathrm{biZar,exc}}$ is meager in
  $Y^{\mathrm{biZar}}$.
\end{lemma}
\begin{proof}
  First we show that $Z$ is not contained in the singular locus of
  $Z^{\mathrm{biZar}}$.
  Indeed, being a singular point of $Z^{\mathrm{biZar}}$ is an
  algebraic condition in $\IA_g$. A component of the preimage
  of $Z^{\mathrm{biZar}}(\IC)$ under
  $\mathfrak{H}_g\rightarrow \IA_g^{\mathrm{an}}$ is algebraic. So being a
  singular point is also an algebraic condition in $\mathfrak{H}_g$.
  Therefore, each irreducible component of  $Z\ssm
  Z^{\mathrm{biZar,reg}}$ is bi-algebraic.
  As $Z^{\mathrm{biZar}}$ is the minimal bi-algebraic subvariety
  containing $Z$ we have $Z\cap Z^{\mathrm{biZar,reg}}\not=\emptyset$.
  The first part of the follows.

  The second claim follows from the first one and since $\pi(Y^{\mathrm{biZar}})=
  \pi(Y)^{\mathrm{biZar}}$.

  The third claim follows from the second one and from 
  Lemma~\ref{lem:meagersuperset} with $X = Y^{\mathrm{biZar}}$.
\end{proof}

\section{A Criterion for Non-degeneracy}
\label{sec:critnondeg}
Recall that $\mathfrak{A}_{g}$ is a geometrically irreducible
quasi-projective variety defined over a number field. Again we take
this number field to be a subfield of $\IC$. For the rest of this
section we
consider all subvarieties as defined over $\IC$. 

Let $X\subset\mathfrak{A}_{g}$ be an irreducible closed subvariety. 
We set
\begin{equation*}
  \delta(X) = \dim X^{\mathrm{biZar}} - \dim \pi(X^{\mathrm{biZar}})
  \ge 0,
\end{equation*}
and with $t\in\IZ$, also
\begin{equation}
  \label{def:degtlocus}
  X^{\mathrm{deg}}(t)=\bigcup_{\substack{Y\subset X \\ \delta(Y) < \dim Y
  + t \\ \dim Y > 0}} Y
\end{equation}
where $Y$ ranges over positive dimensional \textit{irreducible} closed
subvarieties of $X$. Thus
\begin{equation*}
   X^{\mathrm{deg}}(t)\subset  X^{\mathrm{deg}}(t+1).
\end{equation*}

By \cite[Theorem~7.1]{GaoBettiRank}, $X^{\mathrm{deg}}(t)$ is Zariski
closed in $X$. Moreover if $X$ is defined over some algebraically closed field $L \subseteq \IC$ of characteristic $0$, then 
$X^{\mathrm{deg}}(t)$ is also defined over $L$; see \cite[Proposition~4.2.6]{GaoHDR}.

\begin{remark}
  Before moving on, let us take a look at $X^{\mathrm{deg}}(t)$ when
  $\pi(X)$ is a point. In this case, $X$ is contained in a fiber of $\pi
  \colon \mathfrak{A}_g \rightarrow \mathbb{A}_g$, which is an abelian
  variety. Call this abelian variety $A$. For each irreducible
  subvariety $Y$ of $X$, we have $\delta(Y) = \dim Y^{\mathrm{biZar}}
  \ge \dim Y$. In particular, $X^{\mathrm{deg}}(t) = \emptyset$ if $t
  \le 0$. 
  
  By \cite[Proposition~5.1]{UYweaklyspecial}, any bi-algebraic
  subvariety
  of $A$
  is a
  coset in $A$, \textit{i.e.}, a translate of an abelian subvariety of
  $A$. Conversely, any coset in $A$ is bi-algebraic.
  Thus $Y^{\mathrm{biZar}}$ is the smallest coset of $A$ containing $Y$. 
  Now if $\delta(Y) < \dim Y+1$, then $Y^{\mathrm{biZar}} =
  Y$.
  Thus $X^{\mathrm{deg}}(1)$ is the union of all positive-dimensional
  cosets in $A$ that are contained in $X$. This is precisely the Ueno
  locus or Kawamata locus.

  For general $t$ and $X$ still in $A$, the union
  $X^{\mathrm{deg}}(t)$ was studied by R\'emond~\cite[$\mathsection
  3$]{RemondInterIII} and by Bombieri, Masser, and Zannier in the
  multiplicative case~\cite{BMZGeometric,BMZUnlikely} under the name
  ($b$-)anomalous.
\end{remark}


We investigate necessary conditions for when $X=X^{\mathrm{deg}}(t)$.
As a general result we mention \cite[Theorem~8.1]{GaoBettiRank} and
the exposition here is heavily motivated by this reference. Our
approach works under the assumption that $X^{\mathrm{deg}}(t)$ contains a
non-empty open subset of $X^{\mathrm{an}}$.

We keep the same setup as introduced in the beginning of $\mathsection$\ref{sec:MT}.

\begin{lemma}
  \label{lem:predegenerate}
  Let $Y\subset \mathfrak{A}_{g}$ be an irreducible closed subvariety
  such that $Y^{\mathrm{exc}}$ is meager in $Y$.
  Let
   $S$ denote the regular locus of  $\pi(Y^{\mathrm{biZar}})$. Then
   $\pi(Y)\cap S\not=\emptyset$ and there exists 
   $\varphi\in
  \mathrm{End}(\mathfrak{A}_{g,S}/S)$ with the following properties
  hold:
  \begin{enumerate}
  \item[(i)] We have
    $\dim \ker\varphi_s= \delta(Y)$ for all $s\in S(\IC)$.
  \item[(ii)] The fiber  $Y^{\mathrm{biZar}}_s$ is
    a finite union of translates of $(\ker\varphi_s)^0$ for all $s\in
    S(\IC)$.
  \item[(iii)] The abelian varieties $\varphi(\mathfrak{A}_{g})_s$ are
    pairwise isomorphic for all $s\in S(\IC)$.
  \item[(iv)] If $\delta(Y)=0$, then $Y$ is a point. 
  \end{enumerate}
\end{lemma}
\begin{proof}
  Recall that $\pi(Y)^{\mathrm{biZar}}$ is
  the smallest bi-algebraic subvariety of $\IA_{g}$ that contains
  $\pi(Y)$ and
  that  it equals $\pi(Y^{\mathrm{biZar}})$.

  By Lemma~\ref{lem:YcapYbizarreg}, $Y^{\mathrm{biZar,exc}}$ is meager
  in $Y^{\mathrm{biZar}}$ and $\pi(Y)\cap S\not=\emptyset$.


  By Proposition~\ref{prop:monodromyaction} applied to
  $Y^{\mathrm{biZar}}$ 
  each fiber of $Y^{\mathrm{biZar}}$ above a complex point of
  $\pi(Y^{\mathrm{biZar}})$ is a finite union of cosets of dimension
  $\delta(Y)$.
  
  We abbreviate $\cA = \pi^{-1}(S)$. 
  We apply Lemma~\ref{lem:abschquot} to the abelian scheme $\cA/S$ and
  the subvariety $Y^{\mathrm{biZar}}\cap \cA$. Let $\varphi$ be the
  endomorphism in the said lemma.

  By the conclusion of Lemma~\ref{lem:abschquot}(ii) we have
  $\dim\ker\varphi_{s} = \dim Y^{\mathrm{biZar}}\cap \cA - \dim
  \pi(Y^{\mathrm{biZar}}\cap \cA)=\delta(Y)$ for all $s\in S(\IC)$.
  Part  (i) now  follows. For later reference we remark that $\varphi$
  is the identity map if 
  $\delta(Y)=0$; see Lemma~\ref{lem:abschquot}.

  By Lemma~\ref{lem:abschquot}(i) we have $\dim
  \varphi(Y^{\mathrm{biZar}}\cap \cA) = \dim \pi(Y^{\mathrm{biZar}}\cap
  \cA)$. 
  By the fiber dimension theorem, the general fiber of
  $\pi|_{\varphi(Y^{\mathrm{biZar}}\cap\cA)} \colon
  \varphi(Y^{\mathrm{biZar}}\cap\cA) \rightarrow
  \pi(Y^{\mathrm{biZar}}\cap\cA)=S$ is finite. For $s$ in a Zariski open
  and non-empty subset of $S$ we have that $\varphi(Y^{\mathrm{biZar}}\cap\cA)_s$ is
  finite. Therefore, $Y^{\mathrm{biZar}}_s$ is contained in a finite
  union of $(\ker\varphi_s)^0$ for such $s$. By dimension reasons, these
  $Y^{\mathrm{biZar}}_s$ are a finite union of $(\ker\varphi_s)^0$ and
  $(\ker\varphi_s)^0+Y^{\mathrm{biZar}}_s=Y^{\mathrm{biZar}}_s$.

  Note that $(\ker\varphi)^0$ is smooth over $S$ with geometrically
  irreducible generic fiber, as it is an abelian
  scheme. Moreover, $Y^{\mathrm{biZar}}\cap \cA$ is
  Zariski open in $Y^{\mathrm{biZar}}$ and thus irreducible. It follows
  from a purely topological consideration that $(\ker\varphi)^0\times_S (Y^{\mathrm{biZar}}\cap\cA)$ is
  irreducible. A Zariski open and non-empty subset is mapped to
  $Y^{\mathrm{biZar}}\cap\cA$ under addition. This continues to hold on
  all of $(\ker\varphi)^0\times_S (Y^{\mathrm{biZar}}\cap\cA)$. Thus
  $(\ker\varphi_s)^0+Y^{\mathrm{biZar}}_s=Y^{\mathrm{biZar}}_s$. By
  dimension reasons
  $Y^{\mathrm{biZar}}_s$ is a finite union of $(\ker\varphi_s)^0$ for
  all $s\in S(\IC)$. Part (ii) follows.
  
  For all $s\in S(\IC)$, the image $\varphi(\mathfrak{A}_{g,s}) =
  \varphi(\mathfrak{A}_g)_s$ is isogenous to
  $\mathfrak{A}_{g,s}/(\ker\varphi_s)^0$. The latter are pairwise
  isomorphic abelian varieties for all $s$ by
  Proposition~\ref{prop:monodromyaction}. By consider the morphism to a
  suitable moduli space we conclude that the $\varphi(\mathfrak{A}_g)_s$
  are indeed pairwise isomorphic. We conclude (iii).

  For the proof of part (iv) we assume that $\delta(Y)=0$. As remarked
  above, $\varphi$ is the identity. Therefore,
  $\mathfrak{A}_{g,s}$ are pairwise isomorphic abelian varieties for
  $s\in S(\IC)$. This implies that $S$ is a point and so is $\pi(Y)$.
  But then $\pi(Y^{\mathrm{biZar}})$ is a
  point. As $0=\delta(Y) = \dim Y^{\mathrm{biZar}} - \dim
  \pi(Y^{\mathrm{biZar}})$ we have that
  $Y^{\mathrm{biZar}}$ is a point. The same holds for $Y$ and this
  completes the proof of (iv). 
\end{proof}

Now we are ready to prove a necessary condition for
$X^{\mathrm{deg}}(t)$ being sufficiently large. The next proposition
relies on the previous lemma and the Baire Category Theorem; recall
that group of endomorphisms of an abelian scheme is at most countably
infinite. 

\begin{proposition}
  \label{prop:deg}
  Let $t\in \IZ$
  and let $X$ be an irreducible closed subvariety of $\mathfrak{A}_g$.
  Let $\eta$ denote the generic point of $\pi(X)$ and let $S$ denote
  the regular locus of $\pi(X)$.
  We suppose
  \begin{enumerate}
  \item [(a)]
    $X^{\mathrm{deg}}(t)$ contains an open and non-empty subset of
    $X^{\mathrm{an}}$
  \item[(b)] and
    $X_\eta$ is not contained in a proper algebraic subgroup of
    $\mathfrak{A}_{g,\eta}$,
  \end{enumerate}
  There exists a set $\cY$ of irreducible closed
  positive dimensional subvarieties 
  of $X$ and $\varphi \in \mathrm{End}(\mathfrak{A}_{g,S}/S)$
  with the following properties for all $Y\in\cY$.
  \begin{enumerate}
  \item[(i)] We have
    $\dim \ker\varphi_s= \delta(Y)$ for all $s\in S(\IC)$.
  \item[(ii)] The fiber  $Y^{\mathrm{biZar}}_s$ is
    a finite union of translates of $(\ker\varphi_s)^0$ for all
    complex points $s$ of a Zariski open and dense subset $\pi(Y)$. 
  \item[(iii)] The abelian varieties $\varphi(\mathfrak{A}_{g})_s$ are
    pairwise isomorphic for all complex points $s$ of a Zariski open and
    dense subset of $\pi(Y)$.
  \item[(iv)]
    We have $\delta(Y) < \dim Y + t$ and $\pi(Y)\cap S\not=\emptyset$.
  \item[(v)]
    The set $Y^{\mathrm{exc}}$ is
    meager in $Y$.
  \end{enumerate}
  Finally, the closure of $\bigcup_{Y\in \cY} Y(\IC)$ in
  $X^{\mathrm{an}}$ has non-empty interior.
\end{proposition}
\begin{proof}
  By hypothesis (b) and Lemma~\ref{lem:alternative} applied to
  $X\subset \mathfrak{A}_{g}$, we have that $X^{\mathrm{exc}}$ is meager
  in $X$. Thus $X^{\mathrm{exc}}\subset \bigcup_{i=1}^\infty X_i(\IC)$
  such that all $X_i\subsetneq X$ are Zariski closed.
  For a similar reason and using Proposition~\ref{prop:excep} there exist
  Zariski closed $S_1,S_2,\ldots\subsetneq \pi(X)$, among them is
  $\pi(X)\ssm S$,  with
  $S^{\mathrm{exc}}\subset \bigcup_{i=1}^\infty S_i(\IC)$. 
  
  By hypothesis the union of all $Y\subset X$ with $\dim Y > 0$ and
  \begin{equation}
    \label{eq:tequalzeroprop}
    \delta(Y) = \dim Y^{\mathrm{biZar}} - \dim \pi(Y^{\mathrm{biZar}})
    < \dim Y  + t
  \end{equation}
  contains a non-empty open subset of $X^{\mathrm{an}}$. Let 
 $\cY$ be the collection of those  $Y$ with
 $Y\not\subset\pi^{-1}(S_i)$ and  $Y\not\subset X_i$  for all $i$.
  There is a set
  $N\subset X(\IC)$, meager in  $X$, such that 
  $N\cup \bigcup_{Y\in \cY} Y(\IC)$ contains a non-empty open
  subset of $X^{\mathrm{an}}$.

  Let $Y\in \cY$ be arbitrary. In particular, $\pi(Y)\cap S\not=\emptyset$.  Set $U_Y=
  \pi(Y)\cap \pi(Y^{\mathrm{biZar}})^{\mathrm{reg}}\cap S$; it is a Zariski open and dense
  subset of $\pi(Y)$ by Lemma~\ref{lem:YcapYbizarreg}.
  Therefore, $U_Y\not\subset S_i$ for all $i$ by the choice of $\cY$.
  The Baire Category Theorem implies
  $U_Y(\IC)\not\subset  \bigcup_{i=1}^\infty S_i(\IC)$, so
  $U_Y(\IC)\not\subset S^{\mathrm{exc}}$.

  By definition  we have
  $Y^{\mathrm{exc}}\subset X^{\mathrm{exc}}$
  and so $Y^{\mathrm{exc}}\subset \bigcup_{i=1}^\infty (Y\cap
  X_i)(\IC)$.
  By the choice of $\cY$ we conclude that $Y^{\mathrm{exc}}$ is meager
  in $Y$. 

  Apply
  Lemma~\ref{lem:predegenerate} to $Y$ and obtain $\varphi_Y$, 
  and restrict $\varphi_Y$ to an endomorphism of the abelian
  scheme $\pi^{-1}(U_Y)$. 
%
  Choose $s\in U_Y(\IC)\ssm S^{\mathrm{exc}}$. 
  Then ${\varphi_Y}|_{\pi^{-1}(s)} \in \mathrm{End}(\mathfrak{A}_{g,s})$   extends to an
  endomorphism of $\mathfrak{A}_{g,S}/S$. This extension is unique and it
  coincides with $\varphi_Y$ on $\pi^{-1}(U_Y)$. We use $\varphi_Y$ to
  denote this endomorphism of $\mathfrak{A}_{g,S}/S$. Note that $\delta(Y) =
  \dim\ker(\varphi_Y)_s$ for all $s\in S(\IC)$.

  Recall that  $N\cup{\bigcup_{Y\in \cY} Y^{\mathrm{an}}}$
  contains a non-empty
  open subset of $ X^{\mathrm{an}}$. We rearrange this union and
  conclude that the said open subset lies in 
  $N\cup \bigcup_{\varphi\in  \mathrm{End}(\mathfrak{A}_{g,S}/S)}
  \overline{D_\varphi}$ where $D_\varphi= 
  {\bigcup_{Y\in \cY: \varphi_Y= \varphi} Y(\IC)}$
  and $\overline{D_\varphi}$ denotes
  the topological closure in $X^{\mathrm{an}}$. 
  
  By the Baire Category Theorem there is $\varphi\in \mathrm{End}(\mathfrak{A}_{g,S}/S)$
  such that $\overline{D_\varphi}$
  has
   non-empty interior in $X^{\mathrm{an}}$. In particular, $D_\varphi$
   is Zariski dense in $X$. 

   We claim that the proposition follows with $\cY$ replaced by $\{Y\in \cY :
   \varphi_Y = \varphi\}$. Indeed, properties (i), (ii), and (iii) follow from
   the corresponding properties of Lemma~\ref{lem:predegenerate} and
   (iv) and (v) follow from the choice of $\cY$
\end{proof}

\begin{remark}
  \label{rmk:bettirank}
  The case $t=0$ is closely linked to large fibers of the Betti map;
  see \cite[$\mathsection$3]{GaoBettiRank} for a definition of the
  Betti map. The Betti map is real analytic and defined locally on
  $X^{\mathrm{reg,an}}$.
  Suppose
  that the
  generic rank of the differential is strictly less than $2\dim X$.
  This is the case if $X$ fails to be non-degenerate in the sense of \cite[Definition~1.5]{DGHUnifML}.   
  Then there is a non-empty  open subset of $X^{\mathrm{an}}$ on which
  the  rank  is pointwise strictly less than $2\dim X$.
  Using the first-named author's Ax--Schanuel Theorem~\cite{GaoMixedAS} for
  $\mathfrak{A}_g$ one can recover that
  $X^{\mathrm{deg}}(0)$ contains a non-empty open subset of
  $X^{\mathrm{an}}$.
  So the hypothesis (a) for $t=0$ in Proposition~\ref{prop:deg} is satisfied.
  See also \cite[Theorem~1.7]{GaoBettiRank} for an equivalence.
\end{remark}


\section{The Zeroth Degeneracy Locus in a Fiber Power}
\label{sec:fiberpower}

We keep the notation from $\mathsection$\ref{sec:critnondeg} and
consider all subvarieties
as defined over $\IC$. 
We study ramifications
of Proposition~\ref{prop:deg} in the case $t=0$ for the
 $m$-fold fiber power
$\mathfrak{A}_{g}^{[m]}$ of  $\pi\colon\mathfrak{A}_g\rightarrow
\IA_g$, here $m\in\IN$. There is a natural morphism
$\mathfrak{A}_{g}^{[m]} \rightarrow \mathfrak{A}_{mg}$ which is the
base change of the modular map $\IA_{g}\rightarrow\IA_{mg}$ that
attaches to an abelian variety its $m$-th power compatible  with the
principal polarization and level structure. It can be shown that
$\IA_g\rightarrow\IA_{mg}$ is a closed immersion. So
$\mathfrak{A}_{g}^{[m]} \rightarrow \mathfrak{A}_{mg}$ is a closed
immersion. We will treat $\mathfrak{A}_{g}^{[m]}$ as a closed
subvariety of $\mathfrak{A}_{mg}$.

By abuse of notation let
$\pi\colon \mathfrak{A}_{mg}\rightarrow \IA_{mg}$ denote the
structure morphism. 

Let $X$ be a Zariski closed subset of an abelian variety $A$ defined
over $\IC$. The stabilizer $\mathrm{Stab}(X)$ of $X$ is the algebraic group
determined by $\{P \in A(\IC) : P+X=X\}$.



\begin{theorem}
  \label{thm:deg0}
  Let $X$ be an irreducible closed subvariety of
  $\mathfrak{A}_g^{[m]}$. Consider $X \subset \mathfrak{A}_{mg}$ and let
  $\eta$ denote the generic point of $\pi(X)\subset\IA_{mg}$. We suppose
  \begin{enumerate}
  \item [(a)]
    $X^{\mathrm{deg}}(0)$ contains an open and non-empty subset of
    $X^{\mathrm{an}}$,
  \item[(b)]
    $X_\eta$ is not contained in a proper algebraic subgroup of
    $\mathfrak{A}_{mg,\eta}$,
  \item[(c)] and
    \begin{equation}
      \label{eq:dimhyp}
      \dim X \le 2m.
    \end{equation}
  \end{enumerate}
  Then the following hold true.
  \begin{enumerate}
  \item [(i)]  There exists a Zariski open and dense subset $U\subset \pi(X)$ such
  that 
   for all $s\in U(\IC)$ the stabilizer $ \mathrm{Stab}(X_s)$ has
   dimension at least $m$.
  \item[(ii)] There is a Zariski dense subset $D\subset \pi(X)(\IC)$ such
  that for all $s\in D$ the stabilizer $\mathrm{Stab}(X_s)$
  contains $E^m$ where $E\subset \mathfrak{A}_{g,s}$ is an elliptic
  curve. 
  \end{enumerate}
\end{theorem}
\begin{proof}
  We apply Proposition~\ref{prop:deg} to $X\subset\mathfrak{A}_{mg}$
  in the case $t=0$
  and obtain $\cY$ and $\varphi$. We write $S$ for the regular locus of
  $\pi(X)\subset\IA_{mg}$ and $\cB$ for the abelian scheme
  $\varphi(\mathfrak{A}_{g,S}^{[m]})$ over $S$, see Lemma~\ref{lem:endabscheme}.
  
  Let $Y\in \cY$. Note that $\delta=\delta(Y)\ge 0$ is independent of $Y$ by 
  Proposition~\ref{prop:deg}(i).

  The generic fiber of $\cB_{\pi^{-1}(Y)}\rightarrow\pi(Y)$ is an abelian variety $B$
  defined over the function field of $\pi(Y)$. 
  By Proposition~\ref{prop:deg}(iii) there is a finite
  extension $L$  of the function field of $\pi(Y)$, such 
  that the base change $B_{L}$ is a constant abelian variety over
  $L$. We have $\dim B_L= mg-\delta$. Let $A_L$ denote the base change
  of the generic fiber of $\mathfrak{A}_{g,\pi^{-1}(Y)}\rightarrow\pi(Y)$. Then $B_L$ is a quotient
  of $A_L^m$.
  Thus $A_L^m\rightarrow B_{L}$ factors
  through $\mathrm{Im}_{L/\IC}(A_L^m)_L$ where
  $\mathrm{Im}_{L/\IC}(\cdot)$ denotes  the $L/\IC$-image of an
  abelian variety defined over $L$, see~\cite{Conrad} for a definition
  and properties.
  Since $A_L^m\rightarrow B_L$ is surjective
  we have $\dim B_L \le \dim     \mathrm{Im}_{L/\IC}(A_L^m) = m\dim \mathrm{Im}_{L/\IC}(A_L)$.
  By
  (\ref{eq:tequalzeroprop}) with $t=0$ we find
  \begin{equation}
    \label{eq:dimBOmega1}
    mg - \dim Y<
    mg-\delta = \dim B_L \le
    m\dim
    \mathrm{Im}_{L/\IC}(A_L),  
  \end{equation}
  
  As $Y\subset X$ we have $ \dim Y \le \dim X$. The hypothesis $\dim X \le 2m$
  combined with (\ref{eq:dimBOmega1}) yields $m(g-2) < m \dim
  \mathrm{Im}_{L/\IC}(A_L)$. We cancel $m$ and 
  obtain
  \begin{equation*}
    \dim \mathrm{Im}_{L/\IC}(A_L)  \ge g-1.  
  \end{equation*}  

  If $\pi(Y)$ is a point, then so is $\pi(Y^{\mathrm{biZar}}) =
  \pi(Y)^{\mathrm{biZar}}$. Again (\ref{eq:tequalzeroprop}) with $t=0$ implies
  $\dim Y^{\mathrm{biZar}} < \dim Y$ which contradicts $Y\subset
  Y^{\mathrm{biZar}}$. So 
  $$\dim \pi(Y)\ge 1.$$
  From this we conclude  $\dim \mathrm{Im}_{L/\IC}(A_L) <
  g$ as otherwise  general fibers of  $\mathfrak{A}_g$ above $\pi(Y)$ would
  be pairwise isomorphic abelian varieties. Thus
  \begin{equation*}
    \dim
    \mathrm{Im}_{L/\IC}(A_L)=
    g-1.
  \end{equation*}

  The canonical morphism $A_L\rightarrow \mathrm{Im}_{L/\IC}(A_L)_L$
  is surjective with connected kernel $E$ as we are in characteristic
  $0$. Here $E$ is an elliptic curve and $\mathrm{Im}_{L/\IC}(E)=0$.
  %
  %
  Recall that $\varphi$ induces a homomorphism $\varphi_L\colon
  A_L^m\rightarrow B_L$ and $B_L$ is a constant abelian
  variety. The composition $E^m \rightarrow A_L^m
  \xrightarrow{\varphi_L} B_L$ factors through $E^m\rightarrow
  \mathrm{Im}_{L/\IC}(E^m)_L=  \mathrm{Im}_{L/\IC}(E)_L^m=0$. Therefore,
  $E^m$ lies in the kernel of $\varphi_L$. In particular, 
  \begin{equation}
    \label{eq:dimkerlb}
    \delta = \dim\ker\varphi_L \ge m. 
  \end{equation}
  
  We fix an irreducible variety $W$ with function field $L$ and a
  quasi-finite dominant morphism $W\rightarrow \pi(Y)$. After replacing
  $W$ by a Zariski open subset we can spread $A_L$ and
  $E$ out to abelian schemes $\cA$ and $\cE$ over $W$, respectively.
  The $j$-invariant of $\cE/W$ is a morphism
  $W\rightarrow\IA^1$. If $\dim W > 1$, then there is an irreducible curve
  $W'\subset W$ on which $j$ is constant. All elliptic curves above
  $W'(\IC)$ are isomorphic over $\IC$. 
  But then
  infinitely many fibers of $\cA_g$ above points of $\pi(\IC)$ are
  pairwise isomorphic. This is impossible and so we have  $\dim W\le
  1$. But $\dim W = \dim \pi(Y)\ge 1$, hence
  \begin{equation}
    \label{eq:dimpiY1}
    \dim \pi(Y)=1.
  \end{equation}

  Recall that $\varphi$ is defined above all but finitely many points
  of the curve $\pi(Y)$. Recall also that $\ker\varphi$ contains the $m$-th power of
  an elliptic curve on the generic fiber. So $\ker\varphi_s$ contains
  the $m$-th power of an elliptic curve in $\mathfrak{A}_{g,s}$ for
  all but finitely many $s\in\pi(Y)(\IC)$.  
  
  We draw the following conclusion from Proposition~\ref{prop:deg}(i)
  and (ii) for a Zariski open and dense $U_Y\subset\pi(Y)$.
  If $s\in U_Y(\IC)$, then $Y_s$ is contained in a finite
  union of translates of $(\ker\varphi_s)^0$. The latter is an algebraic
  group of dimension $\delta$. Let $P \in \pi|_Y^{-1}(U_Y)(\IC)$ with
  $\pi(P)=s$. Any irreducible component $C$ of $Y_s$ containing $P$ has
  dimension at least $\dim Y - \dim\pi(Y)$. So $\dim C\ge \dim Y -1 \ge
  \delta$ by (\ref{eq:dimpiY1}) and (\ref{eq:tequalzeroprop}) with $t=0$. But
  $C\subset Y_s$, so $C$ is contained in a translate of $(\ker
  \varphi_s)^0$. Thus $C = P +(\ker\varphi)^0_{s} \subset X$. We
  conclude
  \begin{equation}
    \label{eq:dimPlb}
    \dim_P \varphi|_{X\cap \mathfrak{A}_{g,S}^{[m]}}^{-1}(\varphi(P)) \ge \delta \quad\text{for
      all}\quad
    P\in \pi|_Y^{-1}(U_Y)(\IC)\text{ and all }Y\in\cY.
  \end{equation}
  

  By possibly removing finitely many points from  $U_Y$
  we may arrange that
  $(\ker\varphi_{\pi(P)})^0$ contains the
  $m$-th power of
  an elliptic curve in $ \mathfrak{A}_{g,\pi(P)}$ for all $P\in
  \pi|_Y^{-1}(U_Y)(\IC)$.  

  We write $D = \bigcup_{Y\in \cY} \pi|_{Y}^{-1}(U_Y(\IC))$. 
  The closure of $D$ in $X^{\mathrm{an}}$ equals the closure of 
  $\bigcup_{Y\in \cY} Y(\IC)$ in $X^{\mathrm{an}}$.
  Indeed, this requires some point-set topology and the fact that
  $\pi|_{Y}^{-1}(U_Y(\IC))$ lies dense in $Y^{\mathrm{an}}$.
  In particular, $D$ is Zariski dense in $X$ by
  Proposition~\ref{prop:deg}.

  So (\ref{eq:dimPlb}) holds on the Zariski dense subset $D$ of $X$.
  Therefore, the dimension inequality
  holds for all $P\in (X\cap\mathfrak{A}_{g,S}^{[m]})(\IC)$ by the semi-continuity theorem
  on fiber dimensions.  

  Each fiber of $\varphi \colon \mathfrak{A}_{g,S}^{[m]}\rightarrow \varphi(\mathfrak{A}_{g,S}^{[m]})$
  is the translate of some $(\ker\varphi_{\pi(P)})^0$, which has
  dimension $\delta$.
  We conclude that if $P\in (X\cap\mathfrak{A}_{g,S}^{[m]})(\IC)$, then
  $\varphi|_X^{-1}(\varphi(P))$ contains $P+(\ker\varphi_{\pi(P)})^0$
  as an irreducible component. So $(\ker\varphi_{\pi(P)})^0$ lies in
  the stabilizer of $X_{\pi(P)}$.
  
  The first claim of the theorem follows from (\ref{eq:dimkerlb}) with
  $U=\pi(X\cap \mathfrak{A}_{g,S}^{[m]}) = S$. 

  The second claim follows as $\ker\varphi_s$ contains the $m$-th power
  of an elliptic curve for all  $s$ in $\pi(D)$, which is Zariski
  dense in $\pi(X)$.
\end{proof}

For an abelian variety $A$ and $m\in \IN$ we define
$D_m \colon A^{m+1}\rightarrow A^m$ to be the Faltings--Zhang morphism
determined by 
$D_m(P_0,\ldots,P_m) = (P_1-P_0,\ldots,P_m-P_0)$.

\begin{lemma}
  \label{lem:stabx}
  Let $A$ be an abelian variety defined over $\IC$ and let $C\subset
  A$ be an irreducible closed subvariety of dimension $1$. Suppose $m\ge 2$ and let $X =
  D_m(C^{m+1})\subset A^m$. If $B$ is an abelian
  subvariety of $A$ with $B^m \subset \mathrm{Stab}(X)$, then $B=0$ or 
  $C$ is a translate of $B$.
\end{lemma}
\begin{proof}
  Let $\varphi \colon A\rightarrow A/B$ denote the quotient
  homomorphism and $\varphi^m \rightarrow A^m/B^m = (A/B)^m$ its
  $m$-th power. We set $Z = \varphi^m(X)$. 
  Then $\dim Z \le \dim X - m\dim B$ as $B^m$ is in the
  stabilizer of $X$.
  We have $\dim X \le C^{m+1}= m+1$. So
  \begin{equation}
    \label{eq:varphiXdimineq}
    0\le \dim Z \le \dim X - m\dim B \le m+1 -m\dim B.
  \end{equation}
  This implies $\dim B \le 1+1/m<2$ as $m\ge 2$.

  Let us suppose $B\not=0$, then $\dim B=1$.
  Hence $\dim Z\le 1$  by (\ref{eq:varphiXdimineq}).
  We have
  \begin{equation*}
    (\varphi(P_1-P_0),\ldots,\varphi(P_m-P_0))= \varphi^m(P_1-P_0,\ldots,P_m-P_0) \in Z(\IC)
  \end{equation*}
  for all $P_0,\ldots,P_m\in C(\IC)$. We fix $P_0$ and let
  $P_1,\ldots,P_m$ vary. As $\dim Z\le 1$ and $m\ge 2$ it
  follows that $\varphi$ is constant on the curve $C$.
  For dimension reasons we conclude that $C$ equals a
  translate of $B$. 
\end{proof}

\begin{corollary}
  \label{cor:fiberpowercurve}
  Let $g\ge 2, m\ge 2,$  and 
  let $X$ be an irreducible closed subvariety of
  $\mathfrak{A}_g^{[m]}$ with $\dim\pi(X) \le m-1$.
  We suppose that for all complex points $s$ of a Zariski open and
  dense subset of $\pi(X)$, the fiber $X_s$
  is of the form $D_{m}(C^{m+1})$ where $C\subset \mathfrak{A}_{g,s}$
  is not contained in the translate of a  proper algebraic subgroup of $\mathfrak{A}_{g,s}$.
  Then $X^{\mathrm{deg}}(0)$ does not contain an non-empty  open
  subset of $X^{\mathrm{an}}$. Moreover, the generic Betti rank on $X$
  is  $2\dim X$ and $X$ is non-degenerate in the sense of \cite[Definition~1.5]{DGHUnifML}.   
\end{corollary}
\begin{proof}
  Let $X_s = D_m(C^{m+1})$ with $C$ as in the hypothesis.
  In particular, $C$ is not equal to the translate of an abelian
  subvariety of $\mathfrak{A}_{g,s}$.
  By Lemma~\ref{lem:stabx}, $\mathrm{Stab}(X_s)$ does not contain the
  $m$-th power of a non-zero abelian subvariety of
  $\mathfrak{A}_{g,s}$. So conclusion (ii) of Theorem~\ref{thm:deg0} cannot
  hold.
  
  Moreover, $X_s$ is not contained in a proper algebraic subgroup of
  $\mathfrak{A}_{g,s}^m$ and this remains true for the generic point
  of $\pi(X)$. Moreover, $\dim X \le \dim\pi(X) + m+1\le
  (m-1)+m+1=2m$.
  So hypotheses (b) and (c) of Theorem~\ref{thm:deg0} hold.
  Therefore, hypothesis (a) cannot hold. This is the first claim of
  the corollary. The second claim follows from Remark~\ref{rmk:bettirank}. 
\end{proof}

\section{The First Degeneracy Locus and the Relative Manin--Mumford Conjecture}
\label{sec:relmm}

In this section we provide an exposition of the proof of
Proposition 11.2 of the first-named author's work \cite{GaoBettiRank}.
We proceed slightly differently and  concentrate our efforts
on subvarieties of the universal family of principally polarized
abelian varieties with suitable level structure.

We keep the notation of $\mathsection$\ref{sec:MT} with an important
additional restriction. Let $g\ge 1$ be an integer and equip $\IA_g$
with suitable level structure. Let $\pi\colon
\mathfrak{A}_g\rightarrow\IA_g$ denote the universal family. In this
section we consider $\mathfrak{A}_g$ and $\IA_g$ as irreducible
quasi-projective varieties defined over a number field $\IQbar$, the
algebraic closure of $\IQ$ in $\IC$.

The set of torsion points  $\mathfrak{A}_{g,\mathrm{tors}}$ is
$\bigcup_{s\in \IA_g(\IC)}
\{ P \in \mathfrak{A}_{g,s}(\IC) : \text{$P$ has finite order}. \}$.

We consider here a variant of the Relative Manin--Mumford Conjecture,
inspired by S. Zhang \cite{zhang1998small} and formulated in work of
Pink \cite{Pink05} as well as Bombieri--Masser--Zannier
\cite{BMZGeometric}. We also refer to Zannier's
book~\cite{ZannierBook} for a formulation. In contrast to the general
case, we retain $\mathfrak{A}_g$ as an ambient group scheme and work
only with varieties defined over $\IQbar$.

The following conjecture depends on the dimension parameter
$g\in\IN$.

\medskip\noindent\textbf{Conjecture RelMM($g$)}.\textit{
  Let $X$ be an irreducible closed subvariety of $\mathfrak{A}_g$
  defined over $\IQbar$ and
  let $\eta\in \pi(X)$ denote the generic point. 
  We assume that $\dim X < g$ and that $X_\eta$ is not contained in a
  proper algebraic subgroup of $\mathfrak{A}_{g,\eta}$. Then
  $X(\IQbar)\cap \mathfrak{A}_{g,\mathrm{tors}}$ is not Zariski dense in $X$.}
\medskip

For curves, this conjecture is known, even over $\IC$, thanks to work
of Masser--Zannier and Corvaja--Masser--Zannier
\cite{MZ:torsionanomalous,MasserZannierTorsionPointOnSqEC,
MASSER2014116, MasserZannierRelMMSimpleSur, CorvajaMasserZannier2018,
MasserZannierRMMoverCurve}. Stoll~\cite{Stoll:simtorsion} proved an
explicit case. For surfaces some results are due to the first-named
author~\cite{hab:weierstrass} and
Corvaja--Tsimerman--Zannier~\cite{CTZ:23}.

The goal of this section is to prove Conjecture RelMM($g$) for all $g$
conditional on the following conjecture. Below, a subscript $\IC$ indicates
base change to $\IC$.

\begin{conjecture}
  \label{conj:relmmXdeg1dense}
  Let $g\in\IN$, let
  $X$ be an irreducible closed subvariety of $\mathfrak{A}_g$
  defined over $\IQbar$. 
  If $\dim X>0$ and 
  $X(\IQbar)\cap \mathfrak{A}_{g,\mathrm{tors}}$ is  Zariski dense in
  $X$, then $X_\IC^{\mathrm{deg}}(1)$ is Zariski dense in $X$.
\end{conjecture}

The goal will be achieved by induction on $g$. The induction step,
which is conditional, is the
following theorem.

\begin{theorem}
  \label{thm:relmminduction}
  Suppose Conjecture \ref{conj:relmmXdeg1dense} holds. Let $g\ge 2$ be
  an integer and suppose Conjecture \textrm{RelMM}($g'$) holds for all
  $g'\in \{1,\ldots,g-1\}$. Then Conjecture {{RelMM($g$)}} holds.
\end{theorem}
\begin{proof}
  Let $X\subset\mathfrak{A}_g$ be an irreducible closed subvariety
  defined over $\IQbar$
  satisfying the hypothesis of Conjecture RelMM($g$). Let $\eta$
  denote the generic point of $\pi(X)$. 
  We assume  $X(\IQbar)\cap\mathfrak{A}_{g,\mathrm{tors}}$ is Zariski
  dense in $X$ and  will derive a contradiction.

  We observe  $\dim X >0$. So 
  Conjecture \ref{conj:relmmXdeg1dense} implies that
  $X_\IC^{\mathrm{deg}}(1)$ is Zariski dense in $X$. (Note that $X$
  satisfies $\dim X < g$ and a condition on $X_\eta$, but the two are not
  required to invoke
   Conjecture~\ref{conj:relmmXdeg1dense}).
  By \cite[Theorem~7.1]{GaoBettiRank},
  $X_\IC^{\mathrm{deg}}(1)$ is Zariski closed in $X$. Thus
  $X_\IC = X_\IC^{\mathrm{deg}}(1)$.

  By Lemma~\ref{lem:alternative} and since $X_\eta$ is not contained
  in a proper algebraic subgroup of $\mathfrak{A}_{g,\eta}$ we conclude
  that $X^{\mathrm{exc}}$ is meager in $X$.

  We may apply Proposition~\ref{prop:deg} to
  $X_\IC\subset\mathfrak{A}_{g,\IC}$ in the case $t=1$, both hypotheses
  (a) and (b) are met as $X$ is as in Conjecture RelMM($g$). We obtain
  $\cY$ and $\varphi \in \mathrm{End}(\mathfrak{A}_{g,S}/S)$ as in the
  proposition with $S$ the regular locus of $\pi(X)$. (We note that
  $\varphi$ is defined over $\IQbar$.) Let $\delta$
  denote the common value of $\delta(Y)$ for $Y\in\cY$.

  Observe that  $\bigcup_{Y\in\cY}Y\cap\mathfrak{A}_{g,S,\IC}$
  lies Zariski dense in $X_\IC$.
  It is harmless to assume that $\dim Y$ are equal for
  all $Y\in\cY$, the same can be assumed for $\dim\pi(Y)$.  
  
  Let $\cB = \varphi(\mathfrak{A}_{g,S})$, which is an abelian scheme
  over $S$ of relative dimension $g'$, say.

  For each $Y\in \cY$ the exceptional locus
  $Y^{\mathrm{exc}}$ is meager in $Y$ by Proposition~\ref{prop:deg}(v).
  So $Y^{\mathrm{biZar,exc}}$ is meager in $Y^{\mathrm{biZar}}$ by Lemma~\ref{lem:YcapYbizarreg}.
  Proposition~\ref{prop:monodromyaction}(i) applied to
  $Y^{\mathrm{biZar}}$ implies that each $Y^{\mathrm{biZar}}_s$ is a finite union of
  $\mathrm{unif}(\{\tau\}\times W)$, for $s=\mathrm{unif}(\tau)$; here
  $W\subset\IR^{2g}$ is a linear subspace defined over $\IQ$ with $\dim W = 2\delta$
  and independent of $s$. Recall that if
  $s\in S(\IC)$, then $Y^{\mathrm{biZar}}_s$ is a finite union of
  translates of $(\ker\varphi_s)^0$, see Proposition~\ref{prop:deg}(ii). 

  For each $s\in S(\IC)$, the endomorphism $\varphi_s$ lifts to a
  linear map $\IR^{2g}\rightarrow\IR^{2g}$ mapping $\IZ^{2g}$ to itself.
  The lift is independent of $s$ and necessarily vanishes on $W$.
  Therefore, $\varphi(Y^{\mathrm{an}}\cap\mathfrak{A}_{g,S}^{\mathrm{an}})$ is the image of
  $\mathrm{unif}(\tilde Y \times\{\text{finite set}\})$
  where $\mathrm{unif}(\tilde Y) = \pi(Y)\cap S(\IC)$.

  The abelian scheme $\cB/S$ may not be principally polarized.
  But its geometric generic fiber is isogenous to a principally
  polarized abelian variety.
  We fix an \'etale  morphism $S_{\mathrm{pp}}\rightarrow S$  with
  base change $\cB'=\cB\times_S S_{\mathrm{pp}}$ as in the diagram 
  (\ref{eq:abschdia}) below. After spreading out and possibly
  replacing $S_{\mathrm{pp}}$ by a Zariski open and dense subset we may arrange that  $\cB'\rightarrow
  \cB_{\mathrm{pp}}$ is a fiberwise an isogeny over $S_{\mathrm{pp}}$
  and $\cB_{\mathrm{pp}}$ is principally polarized.
  But now the level structure from $\mathfrak{A}_{g,S}$ may have been
  lost under the isogeny.  
  To remedy this we fix yet another
  \'etale morphism $S_{\mathrm{pp,ls}}\rightarrow
  S_{\mathrm{pp}}$ and do a base change to add suitable torsion sections to
  $\cB_{\mathbb{pp}}$ and ultimately  obtain suitable level structure. This does not
  affect the principal polarization.
  Thus we get a principally polarized abelian scheme
  $\cB_{\mathrm{pp,ls}}/S_{\mathrm{pp,ls}}$ with suitable level
  structure. Its relative dimension equals the relative dimension of
  $\cB/S$. We obtain a Cartesian diagram into
  the corresponding fine moduli space as in the right of the following
  commutative diagram
  \begin{equation}
    \label{eq:abschdia}
    \xymatrix{
\mathfrak{A}_{g,S} \ar[d]\ar[r]^{\varphi}&    \cB   \ar[d]
&\ar[l]\leftpullback \cB'\ar[d]\ar[r] & \cB_{\mathrm{pp}}  \ar[d]
&\ar[l]\leftpullback \cB_{\mathrm{pp,ls}}\ar[d] \ar[r]\pullbackcorner & \mathfrak{A}_{g'} \ar[d]\\
S\ar@{=}[r]&      S  &\ar[l] S_{\mathrm{pp}}
\ar@{=}[r]&S_{\mathrm{pp}} &\ar[l] S_{\mathrm{pp,ls}} \ar[r] 
&\IA_{g'}
    }
  \end{equation}
  
  We chase $\varphi(X\cap\mathfrak{A}_{g,S})\subset \cB$
  through the  correspondences
  $\cB\leftarrow \cB'\rightarrow\cB_{\mathrm{pp}}$ and $\cB_{\mathrm{pp}}\leftarrow
  \cB_{\mathrm{pp,ls}}\rightarrow
  \mathfrak{A}_{g'}$
  by taking preimages and images and fix an irreducible component $X'$ of the
  Zariski closure of the image inside $\mathfrak{A}_{g'}$.
  Consider $Y\in \cY$ and chase $\varphi(Y\cap\mathfrak{A}_{g,S,\IC})\subset\cB$
  through the diagram as just described. Recall that
  $\varphi(Y(\IC)\cap\mathfrak{A}_{g,S}(\IC))$ is the image of
  $\mathrm{unif}(\tilde Y \times\{\text{finite set}\})$.
  Locally in the Euclidean topology on the base, our abelian schemes
  are trivializable in the real analytic category.
  Moreover, all abelian
  varieties of $\cB$ above $S(\IC)\cap\pi(Y(\IC))$ are isomorphic by
  Proposition~\ref{prop:monodromyaction}(ii). So
  $\varphi(Y(\IC)\cap\mathfrak{A}_{g,S}(\IC))$ ends up
  as a finite set in $\mathfrak{A}_{g'}$. Thus applying both
  correspondence has
  fibers of dimension at least $\dim \varphi(Y\cap\mathfrak{A}_{g,S,\IC})
  = \dim \pi(Y)$. 
  Thus $\dim X' \le \dim \varphi(X\cap \mathfrak{A}_{g,S,\IC}) -
  \dim \pi(Y)$ by analysis of the fibers of the two correspondences.
  
  Note
  $\dim\varphi(X\cap \mathfrak{A}_{g,S}) \le \dim X - (\dim Y -
  \dim\pi(Y))$ because all fibers of $Y\rightarrow\pi(Y)$ have finite
  image under $\varphi$.
  We find $\dim X' \le \dim X
  - \dim Y < g -\dim Y$, having used $\dim X<g$.

  The relative dimension of
  $\cB/S$ is $g'=g-\delta$. As all elements in $\cY$ have positive
  dimension, Lemma~\ref{lem:predegenerate}(iv) applied to an element
  in $\cY$ implies $\delta \ge 1$.
  Therefore, $g'\le g-1$. 
  We have further  $\delta \le \dim Y$  by Proposition~\ref{prop:deg}(iv) with
  $t=1$. We conclude $\dim X' < g-\dim Y \le g-\delta =g'$. In
  particular, $g'\ge 1$.

  Chasing the Zariski dense set of torsion points in $X(\IQbar)\cap
  \mathfrak{A}_{g,\mathrm{tors}}$ through the  diagram shows that
  the torsion points in 
  $X'(\IQbar)$ are Zariski dense in $X'$.
  The generic fiber of
  $\varphi(X\cap\mathfrak{A}_{g,S})\rightarrow S$ is
  not contained in a proper algebraic subgroup of the generic fiber of
  $\cB\rightarrow S$ by the hypothesis on $X$ in RelMM($g$). This
  implies that the  generic fiber of $X'\rightarrow \pi(X')$ is not
  contained in a proper algebraic subgroup of the generic fiber of
  $\mathfrak{A}_{g',\pi^{-1}(X')}\rightarrow\pi(X')$.

  Recall  
  that $g' \in \{1,\ldots,g-1\}$ and so RelMM($g'$) holds by
  hypothesis. But then the properties of $X'$ contradict the
  conclusion of RelMM($g'$).   
\end{proof}

\begin{corollary}
  \label{cor:relmmconditional}
  Conjecture \ref{conj:relmmXdeg1dense}
  implies  Conjecture RelMM($g$)  for all $g\in\IN$.
\end{corollary}
\begin{proof}
  By Theorem~\ref{thm:relmminduction} it suffices to prove RelMM($1$).
  The condition on $\dim X$ in Conjecture RelMM($1$)
  implies that $X$ is a point. The condition on $X_\eta$ and $g=1$ imply that $X$
  is not a torsion point. So Conjecture RelMM($1$) holds true.   
\end{proof}

\bibliographystyle{alpha}
\bibliography{literature}


\end{document}